\documentclass[reqno,10pt,centertags,draft]{amsart}
\usepackage{amsmath,amsthm,amscd,amssymb,latexsym,esint,upref,enumerate,color,
verbatim,yfonts}
\date{\today}

\makeatletter
\def\theequation{\@arabic\c@equation}

\newcommand{\bb}[1]{{\mathbb{#1}}}
\newcommand{\mc}[1]{{\mathcal{#1}}}
\newcommand{\oT}{H}
\newcommand{\D}{\mathbb{D}}
\newcommand{\bbN}{{\mathbb{N}}}
\newcommand{\bbR}{{\mathbb{R}}}

\newcommand{\bbC}{{\mathbb{C}}}

\newcommand{\cB}{{\mathcal B}}

\newcommand{\cD}{{\mathcal D}}
\newcommand{\cE}{{\mathcal E}}

\newcommand{\cH}{{\mathcal H}}

\newcommand{\cK}{{\mathcal K}}

\newcommand{\cM}{{\mathcal M}}

\newcommand{\cX}{{\mathcal X}}

\newcommand{\no}{\nonumber}
\newcommand{\lb}{\label}
\newcommand{\f}{\frac}

\newcommand{\ol}{\overline}
\newcommand{\bs}{\backslash}

\newcommand{\wti}{\widetilde}

\newcommand{\Oh}{O}

\newcommand{\loc}{\text{\rm{loc}}}

\newcommand{\ran}{\operatorname{ran}}

\newcommand{\dom}{\operatorname{dom}}
\newcommand{\supp}{\operatorname{supp}}

\renewcommand{\Re}{\operatorname{Re}}
\renewcommand{\Im}{\operatorname{Im}}

\renewcommand{\oint}{\ointctrclockwise}
\newcommand{\e}{\hbox{\rm e}}

\newcommand{\bi}{\bibitem}

\newcommand{\dD}{{\partial\hspace*{.2mm}\mathbb{D}}}

\DeclareMathOperator*{\slim}{s-lim}
\DeclareMathOperator*{\wlim}{w-lim}

\numberwithin{equation}{section}

\newtheorem{theorem}{Theorem}[section]

\newtheorem{lemma}[theorem]{Lemma}
\newtheorem{corollary}[theorem]{Corollary}
\newtheorem{hypothesis}[theorem]{Hypothesis}

\theoremstyle{definition}
\newtheorem{definition}[theorem]{Definition}
\newtheorem{remark}[theorem]{Remark}

\begin{document}

\title[Weyl--Titchmarsh Theory and Operator-Valued Potentials]{Initial Value Problems 
and Weyl--Titchmarsh Theory for Schr\"odinger Operators with Operator-Valued 
Potentials}
\author[F.\ Gesztesy, R.\ Weikard, and M.\ Zinchenko]{Fritz
Gesztesy, Rudi Weikard, and Maxim Zinchenko}

\address{Department of Mathematics,
University of Missouri, Columbia, MO 65211, USA}
\email{gesztesyf@missouri.edu}
\urladdr{http://www.math.missouri.edu/personnel/faculty/gesztesyf.html}

\address{Department of Mathematics, University of
Alabama at Birmingham, Birmingham, AL 35294, USA}
\email{rudi@math.uab.edu}
\urladdr{http://www.math.math.uab.edu/\~{}rudi/index.html}

\address{Department of Mathematics,
University of Central Florida, Orlando, FL 32816, USA}
\email{maxim@math.ucf.edu}
\urladdr{http://www.math.ucf.edu/~maxim/}

\date{\today}
\thanks{Based upon work partially supported by the US National Science
Foundation under Grant No.\ DMS-DMS 0965411.}

\begin{abstract}
We develop Weyl--Titchmarsh theory for self-adjoint Schr\"odinger operators 
$H_{\alpha}$ in $L^2((a,b);dx;\cH)$ associated with the operator-valued differential 
expression $\tau =-(d^2/dx^2)+V(\cdot)$, with $V:(a,b)\to\cB(\cH)$, and $\cH$ a 
complex, separable Hilbert space. We assume regularity of the left endpoint $a$ and 
the limit point case at the right endpoint $b$. In addition, the bounded self-adjoint 
operator $\alpha= \alpha^* \in \cB(\cH)$ is used to parametrize the self-adjoint 
boundary condition at the left endpoint $a$ of the type 
$$
\sin(\alpha)u'(a)+\cos(\alpha)u(a)=0, 
$$ 
with $u$ lying in the domain of the underlying maximal operator $H_{\max}$ in 
$L^2((a,b);dx;\cH)$ associated with $\tau$. More precisely, we establish the 
existence of the Weyl--Titchmarsh solution of $H_{\alpha}$, the corresponding 
Weyl--Titchmarsh $m$-function $m_{\alpha}$ and its Herglotz property, and 
determine the structure of the Green's function of $H_{\alpha}$.

Developing Weyl--Titchmarsh theory requires control over certain (operator-valued) solutions of appropriate initial value problems. Thus, we consider existence and 
uniqueness of solutions of 2nd-order differential equations with the operator coefficient 
$V$, 
$$
\begin{cases}
- y'' + (V - z) y = f \, \text{ on } \, (a,b),  \\
\, y(x_0) = h_0, \; y'(x_0) = h_1,  
\end{cases}    
$$
under the following general assumptions:  $(a,b)\subseteq\bbR$ is a finite 
or infinite interval, $x_0\in(a,b)$, $z\in\bbC$, $V:(a,b)\to\cB(\cH)$ is a weakly 
measurable operator-valued function with 
$\|V(\cdot)\|_{\cB(\cH)}\in L^1_\loc((a,b);dx)$, and $f\in L^1_{\loc}((a,b);dx;\cH)$, 
with $\cH$ a complex, separable Hilbert space. We also study the analog of 
this initial value problem with $y$ and $f$ replaced by operator-valued 
functions $Y, F \in \cB(\cH)$. 

Our hypotheses on the local behavior of $V$ appear to be the most general ones 
to date. 
\end{abstract}

\maketitle


\section{Introduction} \lb{s1}

The principal purpose of this paper is to derive a streamlined version of 
Weyl--Titchmarsh theory for Schr\"odinger operators on a finite or infinite 
interval $(a,b) \subset \bbR$ with operator-valued potentials $V \in \cB(\cH)$ 
($\cH$ a complex, separable Hilbert space and $\cB(\cH)$ the Banach space of 
bounded linear operators defined on $\cH$) under very general conditions on 
the local behavior of $V$. We will work under the (simplifying) hypothesis that the 
underlying operator-valued differential expression 
\begin{equation}
\tau = - d^2/dx^2 + V(x), \quad x \in (a,b),    \lb{1.1}
\end{equation}
is regular at the left endpoint $a$ and in the limit point case at the right endpoint 
$b$. (For simplicity, the reader may think of the standard half-line case 
$(a,b)=(0,\infty)$.) 

In performing this task, it is necessary to first study existence and uniqueness 
questions of the following initial value problems associated with $\tau$ in great detail. 
More precisely, in Section \ref{s2} we investigate the following two types of initial value problems: First, we consider existence and uniqueness of $\cH$-valued
solutions $y(z,\cdot,x_0)\in W^{2,1}_\loc((a,b);dx;\cH)$ of the initial value problem
\begin{equation}
\begin{cases}
- y'' + (V - z) y = f \, \text{ on } \, (a,b)\bs E,  \\
\, y(x_0) = h_0, \; y'(x_0) = h_1,
\end{cases}     \lb{1.2}
\end{equation}
where the exceptional set $E$ is of Lebesgue measure zero and independent
of $z$. Here we suppose that $(a,b)\subseteq\bbR$ is a finite or infinite interval, 
$x_0\in(a,b)$, $z\in\bbC$, 
$V:(a,b)\to\cB(\cH)$ is a weakly measurable operator-valued function with
$\|V(\cdot)\|_{\cB(\cH)}\in L^1_\loc((a,b);dx)$, and that $h_0,h_1\in\cH$, and 
$f\in L^1_{\loc}((a,b);dx;\cH)$. 

\smallskip 

\noindent 
In particular, we prove for fixed $x_0,x\in(a,b)$ and $z\in\bbC$, that \\ 
$\bullet$ $y(z,x,x_0)$ 
depends jointly continuously on $h_0,h_1\in\cH$, and $f\in L^1_{\loc}((a,b);dx;\cH)$, \\
$\bullet$  $y(z,x,x_0)$ is strongly continuously differentiable with respect to $x$ 
on $(a,b)$, \\
$\bullet$ $y'(z,x,x_0)$ is strongly differentiable with respect to $x$ on $(a,b)\bs E$, 
\\
and that \\
$\bullet$ for fixed $x_0,x \in (a,b)$, $y(z,x,x_0)$ and $y'(z,x,x_0)$ are entire with 
respect to $z$.

\medskip

Second, again assuming $(a,b)\subseteq\bbR$ to be a finite or infinite interval, 
$x_0\in(a,b)$, $z\in\bbC$, $Y_0,\,Y_1\in\cB(\cH)$, and $F,\,V:(a,b)\to\cB(\cH)$ two 
weakly measurable operator-valued functions with
$\|V(\cdot)\|_{\cB(\cH)},\,\|F(\cdot)\|_{\cB(\cH)}\in L^1_\loc((a,b);dx)$, 
we consider existence and uniqueness of $\cB(\cH)$-valued solutions 
$Y(z,\cdot,x_0):(a,b)\to\cB(\cH)$ of the initial value problem
\begin{equation}
\begin{cases}
- Y'' + (V - z)Y = F \, \text{ on } \, (a,b)\bs E,  \\
\, Y(x_0) = Y_0, \; Y'(x_0) = Y_1, 
\end{cases} \lb{1.3}
\end{equation}
where again the exceptional set $E$ is of Lebesgue measure zero and independent
of $z$. 

\smallskip 

\noindent 
For fixed $x_0 \in (a,b)$ and $z \in \bbC$, we prove that \\
$\bullet$ $Y(z,x,x_0)$ is continuously differentiable with respect to $x$ on $(a,b)$ 
in the $\cB(\cH)$-norm,  \\
$\bullet$ $Y'(z,x,x_0)$ is strongly differentiable with respect to $x$ on $(a,b)\bs E$, \\
and that \\
$\bullet$ for fixed $x_0, x \in (a,b)$, $Y(z,x,x_0)$ and $Y'(z,x,x_0)$ are entire in $z$ in
the $\cB(\cH)$-norm.

In addition, Section \ref{s2} introduces the notion of regular endpoints of intervals, several 
notions of Wronskians, the variation of constants formula, and several versions of Green's formula. 

Our principal Section \ref{s3} then develops Weyl--Titchmarsh theory associated with the operator-valued differential expression $\tau$ in \eqref{1.1} under the simplifying (yet 
most important) assumption that the left endpoint $a$ is regular for $\tau$ and that the 
right endpoint $b$ is of the limit point type for $\tau$. We introduce minimal and maximal operators associated with $\tau$, show that they are adjoint to each other, introduce the 
self-adjoint operators $H_{\alpha}$ in the underlying Hilbert space 
$L^2((a,b);dx;\cH)$, parametrized by the bounded self-adjoint 
operator $\alpha= \alpha^* \in \cB(\cH)$ in the self-adjoint boundary condition at the left endpoint $a$ of the type 
$$
\sin(\alpha)u'(a)+\cos(\alpha)u(a)=0, 
$$ 
with $u$ lying in the domain of the maximal operator $H_{\max}$ in 
$L^2((a,b);dx;\cH)$, establish the existence of the Weyl--Titchmarsh solution of 
$H_{\alpha}$, introduce the corresponding Weyl--Titchmarsh $m$-function 
$m_{\alpha}$ and its Herglotz property, and determine the structure of the Green's 
function of $H_{\alpha}$.

Appendix \ref{sA} then establishes basic facts on bounded operator-valued Herglotz functions (i.e., $\cB(\cH)$-valued functions $M$ analytic in the open upper complex 
half-plane $\bbC_+$ with $\Im(M(\cdot) \geq 0$ on $\bbC_+$).

While we restrict our attention to the case $(a,b)$ with $a$ a regular point for 
$\tau$ and $\tau$ in the limit point case at $b$, it is clear  
how to apply the standard $2 \times 2$ block operator formalism (familiar in the 
case of scalar and matrix-valued potentials $V$) to obtain the Weyl--Titchmarsh 
formalism for Schr\"odinger operators with both endpoints $a$ and $b$ in the limit 
point case (and hence Schr\"odinger operators on the whole real line $\bbR$, 
cf.\ Remark \ref{r3.19}). 

Of course, Schr\"odinger operators with bounded and unbounded operator-valued 
potentials $V(\cdot)$ have been studied in the past and we will briefly review the fundamental contributions in this area next. We note, however, that our hypotheses 
on the local behavior of $V(\cdot) \in \cB(\cH)$ appear to be the most general to 
date. 

The case of Schr\"odinger operators with operator-valued potentials under various 
continuity or smoothness hypotheses on $V(\cdot)$ and under various self-adjoint 
boundary conditions on bounded and unbounded open intervals received considerable attention in the past: 
In the special case where $\dim(\cH)<\infty$, that is, in the case of Schr\"odinger 
operators with matrix-valued potentials, the literature is so voluminous that we cannot 
possibly describe individual references and hence we primarily refer to 
\cite{AM63}, \cite{RK05}, and the references cited therein. We also mention that the 
finite-dimensional case, $\dim(\cH) < \infty$, as discussed in \cite{BL00}, is of 
considerable interest as it represents an important ingredient in some proofs of 
Lieb--Thirring inequalities (cf.\ \cite{LW00}). 

In addition, the constant coefficient case, where $\tau$ is of the form 
$\tau = - (d^2/dx^2) + A$, has received overwhelming attention. But since this is not 
the focus of this paper we just refer to \cite{GG73}, \cite[Chs.\ 3, 4]{GG91}, 
\cite{MN11a}, and the literature cited therein. 

In the particular case of Schr\"odinger-type operators corresponding to the differential expression $\tau = - (d^2/dx^2) + A + V(x)$ on a bounded interval 
$(a,b) \subset \bbR$ with either 
$A=0$ or $A$ a self-adjoint operator satisfying $A\geq c I_{\cH}$ for some $c>0$, 
unique solvability of boundary value problems, the asymptotic behavior of 
eigenvalues, and trace formulas 
in connection with various self-adjoint realizations of $\tau = - (d^2/dx^2) + A + V(x)$ 
on a bounded interval $(a,b)$ are discussed, for instance, in \cite{ABK91}, 
\cite{ABK93}
\cite{AK92}, \cite{As08}, 
\cite{Go68}, \cite{GG69}, \cite{GM76}, \cite{Gu08}, \cite{Mo07}, \cite{Mo10} (for the 
case of spectral parameter dependent separated boundary conditions, see also  
\cite{Al06}, \cite{Al10}, \cite{BA10}).   

For earlier results on various aspects of boundary value problems, spectral theory, 
and scattering theory in the half-line case $(a,b) =(0,\infty)$, the situation closely 
related to the principal topic of this paper, we refer, for instance, 
to \cite{Al06a}, \cite{AM10}, \cite{De08}, \cite{Go68}, \cite{Go71}, \cite{GG69}, \cite{GM76}, 
\cite{KL67}, \cite{Mo07}, \cite{Mo10}, \cite{Ro60}, \cite{Sa71}, \cite{Tr00} 
(the case of the real line 
is discussed in \cite{VG70}). While our treatment of initial value problems was inspired 
by the one in \cite{Sa71}, we permit a more general local behavior of $V(\cdot)$. In addition, we also put particular emphasis on Weyl--Titchmarsh theory and the 
structure of the Green's function of $H_{\alpha}$. 

We should also add that this paper represents a first step in our program. Step two 
will be devoted to spectral properties of $H_{\alpha}$, and step three will aim at 
certain classes of unbounded operator-valued potentials $V$, applicable to 
multi-dimensional Schr\"odinger operators in $L^2(\bbR^n; d^n x)$, $n \in \bbN$, 
$n \geq 2$, generated by differential expressions of the type $\Delta + V(\cdot)$. In 
fact, it was precisely the connection between multi-dimensional Schr\"odinger 
operators and one-dimensional Schr\"odinger operators with unbounded operator-valued potentials which originally motivated our interest in this program. This connection was already employed by Kato \cite{Ka59} in 1959; for more recent applications of this connection between one-dimensional Schr\"odinger operators 
with unbounded operator-valued potentials and multi-dimensional Schr\"odinger operators we refer, for instance, to \cite{ACH99}, \cite{Cr01}, 
\cite{Ja70}, \cite{LNS05}, \cite{MN11a}, \cite{Mi76}, \cite{Mi83}, \cite{Mi83a}, 
\cite{Sa08}, \cite{SS07}, \cite{Sa71a}--\cite{Sa79}, and the references cited therein. 
 
Finally, we comment on the notation used in this paper: Throughout, $\cH$ 
denotes a separable, complex Hilbert space with inner product and norm 
denoted by $(\cdot,\cdot)_{\cH}$ (linear in the second argument) and
$\|\cdot \|_{\cH}$, respectively. The identity operator in $\cH$ is written as 
$I_{\cH}$. We denote by $\cB(\cH)$ the Banach space of linear bounded 
operators in $\cH$. The domain, range, kernel (null space)
of a linear operator will be denoted by $\dom(\cdot)$,
$\ran(\cdot)$, $\ker(\cdot)$, 
respectively. The closure of a closable operator $S$ in $\cH$ is denoted by $\ol S$.

\section{The Initial Value Problem of Second-Order Differential Equations with Operator Coefficients} \label{s2}

In this section we provide some basic results about initial value problems for second-order differential equations of the form $-y''+Qy=f$ on an arbitrary open interval
$(a,b) \subseteq \bbR$ with a bounded operator-valued coefficient $Q$, that is, when $Q(x)$ is a bounded operator on a separable, complex Hilbert space $\cH$ for a.e.\ $x\in(a,b)$. In fact, we are interested in two types of situations: In the first one $f(x)$ is an element of the Hilbert space $\cH$ for a.e.\ $x\in (a,b)$, and the solution sought is to take values in $\cH$. In the second situation, $f(x)$ is a bounded operator on $\cH$ for a.e.\ $x\in(a,b)$, as is the proposed solution $y$.

We start with some preliminaries:
Let $(a,b) \subseteq \bbR$ be a finite or infinite interval and $\cX$ a Banach space.
Unless explicitly stated otherwise (such as in the context of operator-valued measures in
Nevanlinna--Herglotz representations, cf., Appendix \ref{sA}), integration of $\cX$-valued functions on $(a,b)$ will
always be understood in the sense of Bochner (cf.\ e.g., \cite[p.\ 6--21]{ABHN01},
\cite[p.\ 44--50]{DU77}, \cite[p.\ 71--86]{HP85}, \cite[Ch.\ III]{Mi78}, \cite[Sect.\ V.5]{Yo80} for 
details). In particular, if $p\ge 1$, the symbol $L^p((a,b);dx;\cX)$ denotes the set of equivalence classes of strongly measurable $\cX$-valued functions which differ at most on sets of Lebesgue measure zero, such that $\|f(\cdot)\|_{\cX}^p \in L^1((a,b);dx)$. The
corresponding norm in $L^p((a,b);dx;\cX)$ is given by
\begin{equation}
\|f\|_{L^p((a,b);dx;\cX)} = \bigg(\int_{(a,b)} dx\, \|f(x)\|_{\cX}^p \bigg)^{1/p}
\end{equation}
and $L^p((a,b);dx;\cX)$ is a Banach space.

If $\cH$ is a separable Hilbert space, then so is $L^2((a,b);dx;\cH)$ (see, e.g., 
\cite[Subsects.\ 4.3.1, 4.3.2]{BW83}, \cite[Sect.\ 7.1]{BS87}). 

One recalls that by a result of Pettis \cite{Pe38}, if $\cX$ is separable, weak 
measurability of $\cX$-valued functions implies their strong measurability.

If $g \in L^1((a,b);dx;\cX)$, $f(x)= \int_{x_0}^x dx' g(x')$, $x_0, x \in (a,b)$, then $f$ is
strongly differentiable a.e.\ on $(a,b)$ and
\begin{equation}
f'(x) = g(x) \, \text{ for a.e.\ $x \in (a,b)$}.
\end{equation}
In addition,
\begin{equation}
\lim_{t\downarrow 0} \f{1}{t} \int_x^{x+t} dx' \|g(x') - g(x)\|_{\cX} = 0 \,
\text{ for a.e.\ $x \in (a,b)$,}
\end{equation}
in particular,
\begin{equation}
\slim_{t\downarrow 0}\f{1}{t} \int_x^{x+t} dx' g(x') = g(x)
\, \text{ for a.e.\ $x \in (a,b)$.}
\end{equation}

Sobolev spaces $W^{n,p}((a,b); dx; \cX)$ for $n\in\bbN$ and $p\geq 1$ are defined as follows: $W^{1,p}((a,b);dx;\cX)$ is the set of all
$f\in L^p((a,b);dx;\cX)$ such that there exists a $g\in L^p((a,b);dx;\cX)$ and an
$x_0\in(a,b)$ such that
\begin{equation}
f(x)=f(x_0)+\int_{x_0}^x dx' \, g(x') \, \text{ for a.e.\ $x \in (a,b)$.}
\end{equation}
In this case $g$ is the strong derivative of $f$, $g=f'$. Similarly,
$W^{n,p}((a,b);dx;\cX)$ is the set of all $f\in L^p((a,b);dx;\cX)$ so that the first $n$ strong
derivatives of $f$ are in $L^p((a,b);dx;\cX)$. For simplicity of notation one also introduces
$W^{0,p}((a,b);dx;\cX)=L^p((a,b);dx;\cX)$. Finally, $W^{n,p}_{\rm loc}((a,b);dx;\cX)$ is
the set of $\cX$-valued functions defined on $(a,b)$ for which the restrictions to any
compact interval $[\alpha,\beta]\subset(a,b)$ are in $W^{n,p}((\alpha,\beta);dx;\cX)$.
In particular, this applies to the case $n=0$ and thus defines $L^p_{\rm loc}((a,b);dx;\cX)$.
If $a$ is finite we may allow $[\alpha,\beta]$ to be a subset of $[a,b)$ and denote the
resulting space by $W^{n,p}_{\rm loc}([a,b);dx;\cX)$ (and again this applies to the case
$n=0$).

Following a frequent practice (cf., e.g., the discussion in \cite[Sect.\ III.1.2]{Am95}), we
will call elements of $W^{1,1} ([c,d];dx;\cX)$, $[c,d] \subset (a,b)$ (resp.,
$W^{1,1}_{\rm loc}((a,b);dx;\cX)$), strongly absolutely continuous $\cX$-valued functions
on $[c,d]$ (resp., strongly locally absolutely continuous $\cX$-valued functions
on $(a,b)$), but caution the reader that unless $\cX$ posseses the Radon--Nikodym 
(RN) property, this notion differs from the classical definition 
of $\cX$-valued absolutely continuous functions (we refer the interested reader 
to \cite[Sect.\ VII.6]{DU77} for an extensive list of conditions equivalent to $X$ having the 
RN property). Here we just mention that reflexivity of $X$ implies the RN property. 

In the special case where $\cX = \bbC$, we omit $\cX$ and just write
$L^p_{(\loc)}((a,b);dx)$, as usual.

{\bf A Remark on notational convention:} To avoid possible confusion later on between 
two standard
notions of strongly continuous operator-valued functions $F(x)$,
$x \in (a,b)$, that is, strong continuity of $F(\cdot) h$ in $\cH$ for all $h \in\cH$ (i.e.,
pointwise continuity of $F(\cdot)$), versus strong continuity of $F(\cdot)$ in the norm
of $\cB(\cH)$ (i.e., uniform continuity of $F(\cdot)$), we will always mean pointwise continuity of $F(\cdot)$ in $\cH$. The same pointwise conventions will apply to the notions of strongly differentiable and strongly measurable operator-valued functions throughout this manuscript.
In particular, and unless explicitly stated otherwise, for operator-valued functions $Y$, the symbol $Y'$ will be understood in the strong sense; similarly,  $y'$ will denote the strong derivative for vector-valued functions $y$.

The following elementary lemma is probably well-known, but since we repeatedly use
it below, and we could not quickly locate it in the literature, we include a detailed proof:

\begin{lemma} \label{l2.1}
Let $(a,b)\subseteq\bbR$. Suppose $Q:(a,b)\to\cB(\cH)$ is a weakly
measurable operator-valued function with $\|Q(\cdot)\|_{\cB(\cH)}\in L^1_\loc((a,b);dx)$ and $g:(a,b)\to\cH$ is $($weakly$)$ measurable. Then $Qg$ is $($strongly$)$ measurable. Moreover, if $g$ is strongly continuous, then there exists a set $E\subset(a,b)$ with zero Lebesgue measure, depending only on $Q$, such that for every $x_0\in(a,b)\bs E$,
\begin{equation}
\lim_{t\downarrow0}\f{1}{t}\int_{x_0}^{x_0+t} dx \, \|Q(x)g(x) - Q(x_0)g(x_0)\|_\cH = 0,
\lb{2.6A}
\end{equation}
in particular,
\begin{equation}
\slim_{t\downarrow0} \f{1}{t}\int_{x_0}^{x_0+t} dx \, Q(x)g(x) = Q(x_0)g(x_0),
\lb{2.7A}
\end{equation}
that is, the set of Lebesgue points of $Q(\cdot)g(\cdot)$ can be
chosen independently of $g$.
\end{lemma}
\begin{proof}
Since by hypothesis, $Q(\cdot)$ on $(a,b)$ is weakly measurable in $\cH$, that is,
\begin{equation}
(f,Q(\cdot) g)_{\cH} \, \text{ is (Lebesgue) measurable for all $f, g \in \cH$,}
\end{equation}
one infers that this is equivalent to $Q(\cdot)^*$ on $(a,b)$ being weakly measurable
in $\cH$. An application of Pettis' theorem \cite{Pe38} then yields that $Q(\cdot) f$ (equivalently, $Q(\cdot)^* f$) on $(a,b)$ is strongly measurable for all $f \in \cH$.

Next, let $\{e_n\}_{n \in\bbN}$ be a complete orthonormal system in $\cH$. Then writing
\begin{equation}
\|Q(\cdot) f\|^2_{\cH} = \sum_{n\in\bbN} (Q(\cdot) f, e_n)_{\cH} (e_n,Q(\cdot) f)_{\cH},
\end{equation}
one concludes that $\|Q(\cdot)f\|_{\cH}$ on $(a,b)$ is measurable for all $f\in\cH$. In
addition, let $h(\cdot)$ on $(a,b)$ be a weakly (and hence, strongly) measurable function
in $\cH$. Then
\begin{equation}
(f, Q(\cdot) h(\cdot))_{\cH} = (Q(\cdot)^* f, h(\cdot))_{\cH} =
\sum_{n\in\bbN} (Q(\cdot)^* f, e_n)_{\cH} (e_n, h(\cdot))_{\cH},
\end{equation}
implies that $Q(\cdot) h(\cdot)$ on $(a,b)$ is weakly measurable in $\cH$. Another
application of Pettis' theorem then yields the strong measurability of
$Q(\cdot) h(\cdot)$ on $(a,b)$ in $\cH$.

Let $E_0\subset(a,b)$ be a set of Lebesgue measure zero such that
every $x_0\in(a,b)\bs E_0$ is a Lebesgue point for the function
$\|Q(\cdot)\|_{\cB(\cH)}$, implying,
\begin{align}
\lim_{t\downarrow0}\f{1}{t}\int_{x_0}^{x_0+t} dx \,
\|Q(x)\|_{\cB(\cH)} = \|Q(x_0)\|_{\cB(\cH)}, \quad
x_0\in(a,b)\bs E_0. \lb{2.11A}
\end{align}

Next, let $\{E_n\}_{n\in\bbN}$ be a sequence of subsets of $(a,b)$ such that
each $E_n$ is of Lebesgue measure zero and every $x_0\in(a,b)\bs E_n$
is a Lebesgue point for the vector-valued function $Q(\cdot) e_n$, that
is,
\begin{align}
\lim_{t\downarrow0}\f{1}{t}\int_{x_0}^{x_0+t} dx \, \|Q(x) e_n -
Q(x_0) e_n \|_{\cH} = 0, \quad x_0\in(a,b)\bs E_n. \lb{2.12A}
\end{align}

In addition, let $E=\bigcup_{n=0}^\infty E_n$, then every $x_0\in(a,b)\bs E$
is a Lebesgue point for $Q(\cdot)g(\cdot)$. Indeed, decomposing $g(x_0)$
with respect to the orthonormal basis $\{e_n\}_{n\in\bbN}$,
\begin{align}
g(x_0)=\sum_{n\in\bbN}g_n(x_0) e_n, \quad
g_n(x_0)=\big(e_n,g(x_0)\big)_\cH, \; n\in\bbN, \lb{2.13A}
\end{align}
and recalling that by Pettis' theorem, $Q g$ is strongly measurable, yields (for $t>0$)
\begin{align}
&\bigg\|\f{1}{t}\int_{x_0}^{x_0+t} dx \, [Q(x)g(x) - Q(x_0)g(x_0)]\bigg\|_\cH
\no \\
& \quad \le \f{1}{t}\int_{x_0}^{x_0+t} dx \, \|Q(x)g(x) - Q(x_0)g(x_0)\|_\cH
\no \\
&\quad \le \f{1}{t}\int_{x_0}^{x_0+t} dx \,
\|Q(x)[g(x)-g(x_0)] \|_\cH +
\f{1}{t}\int_{x_0}^{x_0+t} dx \,
\|[Q(x)-Q(x_0)] g(x_0)\|_\cH
\no \\
&\quad \le
\f{1}{t}\int_{x_0}^{x_0+t} dx \, \|Q(x)\|_{\cB(\cH)} \,
\sup_{x\in[x_0,x_0+t]}\|g(x)-g(x_0)\|_\cH   \no \\
&\qquad + \sum_{n=1}^{N} |g_n(x_0)|
\bigg(\f{1}{t}\int_{x_0}^{x_0+t} dx \, \|[Q(x)-Q(x_0)] e_n\|_\cH\bigg)
\no \\
&\qquad
+ \bigg(\f{1}{t}\int_{x_0}^{x_0+t} dx \,
[\|Q(x)\|_{\cB(\cH)} + \|Q(x_0)\|_{\cB(\cH)}]\bigg) \,
\bigg\|\sum_{n=N+1}^\infty g_n(x_0) e_n\bigg\|_\cH.    \lb{2.14A}
\end{align}
Finally, taking the limit $t\downarrow 0$ renders the
first term on the right-hand side of \eqref{2.14A} zero as
$g(\cdot)$ is strongly continuous in $\cH$ and $x_0$ is a Lebesgue point
of $\|Q(\cdot)\|_{\cB(\cH)}$ by \eqref{2.11A}. Similarly, taking $t\downarrow 0$ renders
the second term on the right-hand side of \eqref{2.14A} zero by \eqref{2.12A}.
Again by \eqref{2.11A}, the third term on the
right-hand side of \eqref{2.14A} approaches
$2\|Q(x_0)\|_{\cB(\cH)}\big\|\sum_{n=N+1}^\infty g_n(x_0) e_n\big\|_\cH$ as
$t\downarrow 0$ and hence vanishes in the limit $N\to\infty$ (cf.\ \eqref{2.13A}).
\end{proof}

In connection with \eqref{2.7A} we also refer to \cite[Theorem\ II.2.9]{DU77}, 
\cite[Subsect.\ III.3.8]{HP85}, \cite[Theorem\ V.5.2]{Yo80}.

\begin{definition} \lb{d2.2}
Let $(a,b)\subseteq\bbR$ be a finite or infinite interval and
$Q:(a,b)\to\cB(\cH)$ a weakly measurable operator-valued function with
$\|Q(\cdot)\|_{\cB(\cH)}\in L^1_\loc((a,b);dx)$, and suppose that
$f\in L^1_{\loc}((a,b);dx;\cH)$. Then the $\cH$-valued function
$y: (a,b)\to \cH$ is called a (strong) solution of
\begin{equation}
- y'' + Q y = f   \lb{2.15A}
\end{equation}
if $y \in W^{2,1}_\loc((a,b);dx;\cH)$ and \eqref{2.15A} holds a.e.\ on $(a,b)$.
\end{definition}

We recall our notational convention that vector-valued solutions of \eqref{2.15A} will always be viewed as strong solutions.

One verifies that $Q:(a,b)\to\cB(\cH)$ satisfies the conditions in
Definition \ref{d2.2} if and only if $Q^*$ does (a fact that will play a role later on, cf.\
the paragraph following \eqref{2.33A}).

\begin{theorem} \lb{t2.3}
Let $(a,b)\subseteq\bbR$ be a finite or infinite interval and
$V:(a,b)\to\cB(\cH)$ a weakly measurable operator-valued function with
$\|V(\cdot)\|_{\cB(\cH)}\in L^1_\loc((a,b);dx)$. Suppose that
$x_0\in(a,b)$, $z\in\bbC$, $h_0,h_1\in\cH$, and $f\in
L^1_{\loc}((a,b);dx;\cH)$. Then there is a unique $\cH$-valued
solution $y(z,\cdot,x_0)\in W^{2,1}_\loc((a,b);dx;\cH)$ of the initial value problem
\begin{equation}
\begin{cases}
- y'' + (V - z) y = f \, \text{ on } \, (a,b)\bs E,  \\
\, y(x_0) = h_0, \; y'(x_0) = h_1,
\end{cases}     \lb{2.1}
\end{equation}
where the exceptional set $E$ is of Lebesgue measure zero and independent
of $z$.

Moreover, the following properties hold:
\begin{enumerate}[$(i)$]
\item For fixed $x_0,x\in(a,b)$ and $z\in\bbC$, $y(z,x,x_0)$ depends jointly continuously on $h_0,h_1\in\cH$, and $f\in L^1_{\loc}((a,b);dx;\cH)$ in the sense that
\begin{align}
\begin{split}
& \big\|y\big(z,x,x_0;h_0,h_1,f\big) - y\big(z,x,x_0;\wti h_0,\wti h_1,\wti f\big)\big\|_{\cH}    \\
& \quad \leq C(z,V)
\big[\big\|h_0 - \wti h_0\big\|_{\cH} + \big\|h_1 - \wti h_1\big\|_{\cH}
+ \big\|f - \wti f\big\|_{L^1([x_0,x];dx;\cH)}\big],    \lb{2.1A}
\end{split}
\end{align}
where $C(z,V)>0$ is a constant, and the dependence of
$y$ on the initial data $h_0, h_1$ and the inhomogeneity $f$ is displayed 
in \eqref{2.1A}. 
\item For fixed $x_0\in(a,b)$ and $z\in\bbC$, $y(z,x,x_0)$ is strongly continuously differentiable with respect to $x$ on $(a,b)$.
\item For fixed $x_0\in(a,b)$ and $z\in\bbC$, $y'(z,x,x_0)$ is strongly differentiable with respect to $x$ on $(a,b)\bs E$.
\item For fixed $x_0,x \in (a,b)$, $y(z,x,x_0)$ and $y'(z,x,x_0)$
are entire with respect to $z$.
\end{enumerate}
\end{theorem}
\begin{proof}
As discussed in the proof of Lemma \ref{l2.1}, if $f:(a,b)\to\cH$ is strongly
measurable, then $Q(\cdot)f(\cdot)$ is also a strongly measurable $\cH$-valued function.

As in the classical scalar case (i.e., $\cH=\bbC$), one can show that a function
$y(z,\cdot,x_0) \in W^{2,1}_\loc((a,b);dx;\cH)$ satisfies the initial-value problem
\eqref{2.1} if and only if $y(z,\cdot,x_0)$ is strongly measurable, strongly locally
bounded, and satisfies the integral equation,
\begin{align}
y(z,x,x_0) &= \cos\big(z^{1/2}(x-x_0)\big)h_0
+ z^{-1/2} \sin\big(z^{1/2}(x-x_0)\big) h_1   \no \\
&\quad + \int_{x_0}^x dx' \, z^{-1/2} \sin\big(z^{1/2}(x-x')\big)
\big[V(x')y(z,x',x_0)-f(x')\big],  \lb{2.1a} \\
& \hspace*{4.2cm} z \in \bbC, \; \Im(z^{1/2}) \geq 0, \; x_0, x \in (a,b).   \no
\end{align}
Thus, it suffices to verify existence and uniqueness for a solution of
\eqref{2.1a}. For uniqueness it is enough to check that $y(z,\cdot,x_0)=0$
is the only solution of
\begin{align}
y(z,x,x_0) &= \int_{x_0}^x dx' \,
z^{-1/2} \sin\big(z^{1/2}(x-x')\big)V(x')y(x'). \lb{2.1b}
\end{align}
Let $K\subset(a,b)$ be a compact subset containing $x_0$, then
iterations of \eqref{2.1b} yield
\begin{align}
\sup_{x\in K}\|y(z,x,x_0)\|_\cH \le
\f{1}{n!}\left(C(z)\int_{x_0}^x\|V(x')\|_{\cB(\cH)}dx'\right)^n
\sup_{x'\in K}\|y(z,x',x_0)\|_\cH, \; n\in\bbN,
\end{align}
for an appropriate constant $C(z)>0$.
Since $K$ and $n$ are arbitrary, the only solution of \eqref{2.1b} is
the zero solution.

To show existence one uses the method of successive approximations.
Define a sequence of vector-valued functions
$y_n(z,\cdot,x_0):(a,b)\to\cH$, $n\in\bbN_0$, by
\begin{align}
y_0(z,x,x_0)&=\cos\big(z^{1/2}(x-x_0)\big)h_0
+ z^{-1/2} \sin\big(z^{1/2}(x-x_0)\big) h_1 \no \\
& \quad - \int_{x_0}^x dx' \, z^{-1/2} \sin\big(z^{1/2}(x-x')\big) f(x'), \no \\
y_n(z,x,x_0)&=\int_{x_0}^x dx' \,
z^{-1/2} \sin\big(z^{1/2}(x-x')\big) V(x')y_{n-1}(z,x',x_0), \quad
n\in\bbN. \lb{2.2}
\end{align}
Then for each $n\in\bbN_0$, it follows inductively that for fixed $x_0\in(a,b)$ and $z\in\bbC$, $y_n(z,x,x_0)$ is strongly locally absolutely continuous with respect to $x$ on $(a,b)$, and for fixed $x_0, x \in (a,b)$, $y_n(z,x,x_0)$, $y_n'(z,x,x_0)$ are entire with respect to $z$. The estimate
\begin{align}
\begin{split}
&\|y_n(z,x,x_0)\|_\cH + \|y_n'(z,x,x_0)\|_\cH
\\
&\quad \le \f{1}{n!}\bigg(C\int_{x_0}^x dx' \, \|V(x')\|_{\cB(\cH)}dx'\bigg)^n
\bigg(\|h_0\|_\cH + \|h_1\|_\cH + \int_{x_0}^x\|f(x')\|_{\cH}\bigg),
\end{split} \lb{2.2a}
\end{align}
holds uniformly in $(z,x)$ on compact subsets of $\bbC\times(a,b)$,
where $C$ depends only on the compact subset of $\bbC\times(a,b)$.
This yields convergence of the series,
\begin{align}
y(z,x,x_0) = \sum_{n=0}^\infty y_n(z,x,x_0), \quad y'(z,x,x_0) =
\sum_{n=0}^\infty y_n'(z,x,x_0)   \lb{2.2b}
\end{align}
with
\begin{align}
\|y(z,x,x_0)\|_{\cH} &\le \exp\bigg(C\int_{x_0}^x dx' \, \|V(x')\|_{\cB(\cH)}\bigg) \no
\\
&\quad \times\bigg(\|h_0\|_{\cB(\cH)} + \|h_1\|_{\cH} +
\int_{x_0}^x dx' \, \|f(x')\|_{\cH} \bigg), \lb{2.2c}
\end{align}
uniformly in $(z,x)$ on compact subsets of $\bbC\times(a,b)$. Then \eqref{2.2},
\eqref{2.2b} imply that $y(z,\cdot,x_0)$ is a solution of the integral equation
\eqref{2.1a}, and \eqref{2.2b}, \eqref{2.2c} yield the properties $(i)$ (taking into account
linearity of \eqref{2.1}) and $(iv)$.

Finally, by \eqref{2.1a}, for each $z\in\bbC$ and a.e.\ $x\in(a,b)$,
\begin{align}
y''(z,x,x_0)=zy(z,x,x_0)-V(x)y(z,x,x_0)+f(x),
\end{align}
and hence
\begin{align}
\begin{split}
y(z,x,x_0) &= \cos\big(z^{1/2}(x-x_0)\big)h_0
+ z^{-1/2} \sin\big(z^{1/2}(x-x_0)\big) h_1
\\
&\quad + \int_{x_0}^x dx' \bigg(\int_{x_0}^{x'} dx'' \,
\big[zy(z,x'',x_0)-V(x'')y(z,x'',x_0)+f(x'')\big]\bigg).
\end{split}
\end{align}
This representation of $y(z,x,x_0)$ combined with Lemma~\ref{l2.1} yields the properties $(ii)$ and $(iii)$. In particular, $y(z,\cdot,x_0)\in W^{2,1}_\loc((a,b);dx;\cH)$ and $y(z,\cdot,x_0)$ is a strong solution of the initial value problem \eqref{2.1}.
\end{proof}

For classical references on initial value problems we refer, for instance, to 
\cite[Chs.\ III, VII]{DK74} and \cite[Ch.\ 10]{Di60}, but we emphasize again that our approach minimizes the smoothness hypotheses on $V$ and $f$.
  
\begin{definition} \lb{d2.4}
Let $(a,b)\subseteq\bbR$ be a finite or infinite interval and assume that
$F,\,Q:(a,b)\to\cB(\cH)$ are two weakly measurable operator-valued functions such
that $\|F(\cdot)\|_{\cB(\cH)},\,\|Q(\cdot)\|_{\cB(\cH)}\in L^1_\loc((a,b);dx)$. Then the
$\cB(\cH)$-valued function $Y:(a,b)\to\cB(\cH)$ is called a solution of
\begin{equation}
- Y'' + Q Y = F   \lb{2.26A}
\end{equation}
if $Y(\cdot)h\in W^{2,1}_\loc((a,b);dx;\cH)$ for every $h\in\cH$ and $-Y''h+QYh=Fh$ holds
a.e.\ on $(a,b)$.
\end{definition}

\begin{corollary} \lb{c2.5}
Let $(a,b)\subseteq\bbR$ be a finite or infinite interval, $x_0\in(a,b)$, $z\in\bbC$, $Y_0,\,Y_1\in\cB(\cH)$, and suppose $F,\,V:(a,b)\to\cB(\cH)$ are two weakly measurable operator-valued functions with
$\|V(\cdot)\|_{\cB(\cH)},\,\|F(\cdot)\|_{\cB(\cH)}\in L^1_\loc((a,b);dx)$. Then there is a
unique $\cB(\cH)$-valued solution $Y(z,\cdot,x_0):(a,b)\to\cB(\cH)$ of the initial value
problem
\begin{equation}
\begin{cases}
- Y'' + (V - z)Y = F \, \text{ on } \, (a,b)\bs E,  \\
\, Y(x_0) = Y_0, \; Y'(x_0) = Y_1.
\end{cases} \lb{2.3}
\end{equation}
where the exceptional set $E$ is of Lebesgue measure zero and independent
of $z$. Moreover, the following properties hold:
\begin{enumerate}[$(i)$]
\item For fixed $x_0 \in (a,b)$ and $z \in \bbC$, $Y(z,x,x_0)$ is continuously
differentiable with respect to $x$ on $(a,b)$ in the $\cB(\cH)$-norm.
\item For fixed $x_0 \in (a,b)$ and $z \in \bbC$, $Y'(z,x,x_0)$ is strongly differentiable with respect to $x$ on $(a,b)\bs E$.
\item For fixed $x_0, x \in (a,b)$, $Y(z,x,x_0)$ and $Y'(z,x,x_0)$ are entire in $z$ in
the $\cB(\cH)$-norm.
\end{enumerate}
\end{corollary}
\begin{proof}
Applying Theorem~\ref{t2.3} to $h_0=Y_0h$, $h_1=Y_1h$, and $f(x)=F(x)h$ with $h\in\cH$ yields a unique vector-valued solution $y_h(z,x,x_0)$. Since $y_h(z,x,x_0)$ depends continuously on $h$ by Theorem \ref{t2.3} $(i)$, this yields a unique operator-valued solution $Y(z,\cdot,x_0):(a,b)\to\cB(\cH)$ of the initial value problem \eqref{2.3}, where $Y(z,x,x_0)h=y_h(z,x,x_0)$ for all $h\in\cH$.

It follows from Theorem~\ref{t2.3}\,$(ii)$ that for fixed $x_0\in(a,b)$, $z\in\bbC$, and every $h\in\cH$, $\|Y(z,\cdot,x_0)h\|_\cH$ is continuous on $(a,b)$ and hence bounded on every compact subset of $(a,b)$. Thus, it follows from the uniform boundedness principle (cf.\
\cite[Thm.\ III.1.3.29]{Ka80}) that $\|Y(z,\cdot,x_0)\|_{\cB(\cH)}$ is bounded on every compact subset of $(a,b)$.

Moreover, Theorem~\ref{t2.3}\,$(ii)$ and $(iii)$ also imply that $Y(z,x,x_0)$ and $Y'(z,x,x_0)$ are differentiable with respect to $x$ in the strong operator topology. Hence, using
\begin{align}
\begin{split}
&Y(z,x,x_0)h = \cos\big(z^{1/2}(x-x_0)\big)Y_0h +
z^{-1/2} \sin\big(z^{1/2}(x-x_0)\big) Y_1h
\\
&\quad + \int_{x_0}^x dx' \bigg(\int_{x_0}^{x'} dx'' \,
[zY(z,x'',x_0)h-V(x'')Y(z,x'',x_0)h+F(x'')h]\bigg).
\end{split} \lb{2.3a}
\end{align}
one computes
\begin{align}
\begin{split}
&\Big\|\f1t[Y(z,x+t,x_0)-Y(z,x,x_0)]h-Y'(z,x,x_0)h\Big\|_\cH
\\
&\quad \le
\Oh(t)\|Y_0\|_{\cB(\cH)}\|h\|_\cH + \Oh(t)\|Y_1\|_{\cB(\cH)}\|h\|_\cH
\\
&\qquad + \f{1}{|t|} \bigg(\int_{x}^{x+t} dx' \bigg(\int_{x}^{x'} dx'' \,
[|z|+\|V(x'')\|_{\cB(\cH)}]\|Y(z,x'',x_0)\|_{\cB(\cH)}\bigg)\bigg) \|h\|_\cH
\\
&\qquad + \f{1}{|t|}\bigg(\int_{x}^{x+t} dx' \bigg(\int_{x}^{x'}
dx'' \, \|F(x'')\|_{\cB(\cH)}\bigg)\bigg) \|h\|_\cH.
\end{split} \lb{2.3b}
\end{align}
Since the right hand-side vanishes as $t\downarrow 0$ uniformly in $h\in\cH$ with
$\|h\|_\cH\le1$, the solution $Y(z,x,x_0)$ is differentiable with respect to $x$ in the
$\cB(\cH)$-norm topology. Similarly one uses \eqref{2.3a} to verify continuity of
$Y'(z,x,x_0)$ with respect to $x$ in the $\cB(\cH)$-norm topology, implying item $(i)$.

Item $(ii)$ follows directly from Theorem \ref{t2.3}\,$(iii)$ with the set $E$ possibly
dependent on $h\in\cH$. To remove the $h$-dependence one chooses an orthonormal
basis $\{e_n\}_{n\in\bbN}\subset\cH$ and let $E_n$ be the corresponding exceptional sets. Then $E=\bigcup_{n=1}^\infty E_n$ can be used as the exceptional set in item $(ii)$.

Finally, by Theorem~\ref{t2.3} $(iv)$, $Y(z,x,x_0)$ and $Y'(z,x,x_0)$ are
entire with respect to $z$ in the strong operator topology and hence by
\cite[Theorem III.1.37]{Ka80} also in the $\cB(\cH)$-topology, implying item $(iii)$.
\end{proof}

Various versions of Theorem \ref{t2.3} and Corollary \ref{c2.5} exist in the literature
under varying assumptions on $V$ and $f, F$. For instance, the case where $V(\cdot)$ is continuous in the $\cB(\cH)$-norm and $F=0$ is discussed in \cite[Theorem\ 6.1.1]{Hi69}.
The case, where $\|V(\cdot)\|_{\cB(\cH} \in L^1_{\loc}([a,c];dx)$ for all $c>a$ and
$F=0$ is discussed in detail in \cite{Sa71} (it appears that a measurability assumption
of $V(\cdot)$ in the $\cB(\cH)$-norm is missing in the basic set of
hypotheses of \cite{Sa71}). Our extension to $V(\cdot)$ weakly measurable and
$\|V(\cdot)\|_{\cB(\cH} \in L^1_{\loc}([a,b);dx)$ may well be the most general one
published to date, but we obviously claim no originality in this context.

\begin{definition} \lb{d2.6}
Pick $c \in (a,b)$.
The endpoint $a$ (resp., $b$) of the interval $(a,b)$ is called {\it regular} for the operator-valued differential expression $- (d^2/dx^2) + Q(\cdot)$ if it is finite and if $Q$ is weakly measurable and $\|Q(\cdot)\|_{\cB(\cH)}\in  L^1_{\loc}([a,c];dx)$ (resp.,
$\|Q(\cdot)\|_{\cB(\cH)}\in  L^1_{\loc}([c,b];dx)$) for some $c\in (a,b)$. Similarly,
$- (d^2/dx^2) + Q(\cdot)$ is called {\it regular at $a$} (resp., {\it regular at $b$}) if
$a$ (resp., $b$) is a regular endpoint for $- (d^2/dx^2) + Q(\cdot)$.
\end{definition}

We note that if $a$ (resp., $b$) is regular for $- (d^2/dx^2) + Q(x)$, one may allow for
$x_0$ to be equal to $a$ (resp., $b$) in the existence and uniqueness Theorem \ref{t2.3}.

If $f_1, f_2$ are strongly continuously differentiable $\cH$-valued functions, we define the Wronskian of $f_1$ and $f_2$ by
\begin{equation}
W_{*}(f_1,f_2)(x)=(f_1(x),f'_2(x))_\cH - (f'_1(x),f_2(x))_\cH,    \lb{2.31A}
\quad x \in (a,b).
\end{equation}
If $f_2$ is an $\cH$-valued solution of $-y''+Qy=0$ and $f_1$ is an $\cH$-valued
solution of $-y''+Q^*y=0$, their Wronskian $W_{*}(f_1,f_2)(x)$ is $x$-independent, that is,
\begin{equation}
\f{d}{dx} W_{*}(f_1,f_2)(x) = 0, \, \text{ for a.e.\ $x \in (a,b)$.}   \lb{2.32A}
\end{equation}
Equation \eqref{2.52A} will show that the right-hand side of \eqref{2.32A} actually
vanishes for all $x \in (a,b)$.

We decided to use the symbol $W_{*}(\cdot,\cdot)$ in \eqref{2.31A} to indicate its
conjugate linear behavior with respect to its first entry.

Similarly, if $F_1,F_2$ are strongly continuously differentiable $\cB(\cH)$-valued
functions, their Wronskian is defined by
\begin{equation}
W(F_1,F_2)(x) = F_1(x) F'_2(x) - F'_1(x) F_2(x), \quad x \in (a,b).    \lb{2.33A}
\end{equation}
Again, if $F_2$ is a $\cB(\cH)$-valued solution of  $-Y''+QY = 0$ and $F_1$ is a
$\cB(\cH)$-valued solution of $-Y'' + Y Q = 0$ (the latter is equivalent to
$- {(Y^{*})}^{\prime\prime} + Q^* Y^* = 0$ and hence can be handled in complete analogy
via Theorem \ref{t2.3} and Corollary \ref{c2.5}, replacing $Q$ by $Q^*$) their Wronskian will be $x$-independent,
\begin{equation}
\f{d}{dx} W(F_1,F_2)(x) = 0 \, \text{ for a.e.\ $x \in (a,b)$.}
\end{equation}

Our main interest is in the case where $V(\cdot)=V(\cdot)^* \in \cB(\cH)$ is self-adjoint,
that is, in the differential equation $\tau \eta=z \eta$, where $\eta$ represents an $\cH$-valued, respectively, $\cB(\cH)$-valued solution (in the sense of Definitions \ref{d2.2},
resp., \ref{d2.4}), and where $\tau$ abbreviates the operator-valued differential expression
\begin{equation} \label{2.4}
\tau =-(d^2/dx^2) + V(\cdot).
\end{equation}
To this end, we now introduce the following basic assumption:

\begin{hypothesis} \lb{h2.7}
Let $(a,b)\subseteq\bbR$, suppose that $V:(a,b)\to\cB(\cH)$ is a weakly
measurable operator-valued function with $\|V(\cdot)\|_{\cB(\cH)}\in L^1_\loc((a,b);dx)$,
and assume that $V(x) = V(x)^*$ for a.e.\ $x \in (a,b)$.
\end{hypothesis}

Moreover, for the remainder of this section we assume that $\alpha \in \cB(\cH)$ is a
self-adjoint operator,
\begin{equation}
\alpha = \alpha^* \in \cB(\cH).      \lb{2.4A}
\end{equation}

Assuming Hypothesis \ref{h2.7} and \eqref{2.4A}, we introduce the standard fundamental systems of operator-valued solutions of $\tau y=zy$ as follows: Since $\alpha$ is a bounded self-adjoint operator, one may define the self-adjoint operators $A=\sin(\alpha)$ and $B=\cos(\alpha)$ via the spectral theorem. One then concludes that
$\sin^2(\alpha) + \cos^2(\alpha) = I_\cH$ and $[\sin\alpha,\cos\alpha]=0$ (here
$[\cdot,\cdot]$ represents the commutator symbol). The spectral theorem implies also
that the spectra of $\sin(\alpha)$ and $\cos(\alpha)$ are contained in $[-1,1]$ and that the spectra of $\sin^2(\alpha)$ and $\cos^2(\alpha)$ are contained in $[0,1]$. Given such an operator $\alpha$ and a point $x_0\in(a,b)$ or a regular endpoint for $\tau$, we now
define $\theta_\alpha(z,\cdot, x_0,), \phi_\alpha(z,\cdot,x_0)$ as those $\cB(\cH)$-valued
solutions of $\tau Y=z Y$ (in the sense of Definition \ref{d2.4}) which satisfy the initial
conditions
\begin{equation}
\theta_\alpha(z,x_0,x_0)=\phi'_\alpha(z,x_0,x_0)=\cos(\alpha), \quad
-\phi_\alpha(z,x_0,x_0)=\theta'_\alpha(z,x_0,x_0)=\sin(\alpha).    \lb{2.5}
\end{equation}

By Corollary 2.5\,$(iii)$, for any fixed $x, x_0\in(a,b)$, the functions
$\theta_{\alpha}(z,x,x_0)$ and $\phi_{\alpha}(z,x,x_0)$ as well as their strong $x$-derivatives are entire with respect to $z$ in the $\cB(\cH)$-norm. The same is true for the functions $z\mapsto\theta_{\alpha}(\ol{z},x,x_0)^*$ and
$z\mapsto\phi_{\alpha}(\ol{z},x,x_0)^*$.

Since $\theta_{\alpha}(\bar z,\cdot,x_0)^*$ and $\phi_{\alpha}(\bar z,\cdot,x_0)^*$ satisfy
the adjoint equation $-Y''+YV=z Y$ and the same initial conditions as $\theta_\alpha$ and
$\phi_\alpha$, respectively, one obtains the following identities from the constancy of Wronskians:
\begin{align}
\theta_{\alpha}' (\bar z,x,x_0)^*\theta_{\alpha} (z,x,x_0)-
\theta_{\alpha} (\bar z,x,x_0)^*\theta_{\alpha}' (z,x,x_0)&=0, \label{2.7f}
\\
\phi_{\alpha}' (\bar z,x,x_0)^*\phi_{\alpha} (z,x,x_0)-
\phi_{\alpha} (\bar z,x,x_0)^*\phi_{\alpha}' (z,x,x_0)&=0, \label{2.7g}
\\
\phi_{\alpha}' (\bar z,x,x_0)^*\theta_{\alpha} (z,x,x_0)-
\phi_{\alpha} (\bar z,x,x_0)^*\theta_{\alpha}' (z,x,x_0)&=I_{\cH}, \label{2.7h}
\\
\theta_{\alpha} (\bar z,x,x_0)^*\phi_{\alpha}' (z,x,x_0)
- \theta_{\alpha}' (\bar z,x,x_0)^*\phi_{\alpha} (z,x,x_0)&=I_{\cH}. \label{2.7i}
\end{align}
Equations \eqref{2.7f}--\eqref{2.7i} are equivalent to the statement that the block operator
\begin{equation}
\Theta_{\alpha}(z,x,x_0)=\begin{pmatrix}\theta_{\alpha}(z,x,x_0)&\phi_{\alpha}(z,x,x_0)\\ \theta_{\alpha}'(z,x,x_0)&\phi_{\alpha}'(z,x,x_0) \end{pmatrix}    \label{2.7ia}
\end{equation}
has a left inverse given by
\begin{equation}
\begin{pmatrix}\phi_{\alpha}'(\bar z,x,x_0)^*&-\phi_{\alpha}(\bar z,x,x_0)^*\\
-\theta_{\alpha}'(\bar z,x,x_0)^*&\theta_{\alpha}(\bar z,x,x_0)^*
\end{pmatrix}.    \label{2.7ib}
\end{equation}
Thus the operator $\Theta_{\alpha}(z,x,x_0)$ is injective. It is also surjective as will be shown next: Let $(f_1,g_1)^\top$ be an arbitrary element of $\cH\oplus\cH$ and let $y$ be an $\cH$-valued solution of the initial value problem
\begin{equation}
\begin{cases} \tau y=zy, \\ y(x_1)=f_1, \; y'(x_1)=g_1, \end{cases}
\end{equation}
for some given $x_1\in(a,b)$. One notes that due to the initial conditions specified in
\eqref{2.5}, $\Theta_{\alpha}(z,x_0,x_0)$ is bijective. We now assume that $(f_0,g_0)^\top$ are given by
\begin{equation}
\Theta_{\alpha}(z,x_0,x_0)\begin{pmatrix}f_0\\ g_0\end{pmatrix}=\begin{pmatrix}y(x_0)\\ y'(x_0)\end{pmatrix}.   \label{2.7ic}
\end{equation}
The existence and uniqueness Theorem \ref{t2.3} then yields that
\begin{equation}
\Theta_{\alpha}(z,x_1,x_0)\begin{pmatrix}f_0\\ g_0\end{pmatrix}=\begin{pmatrix}f_1\\ g_1\end{pmatrix}.    \label{2.7id}
\end{equation}
This establishes surjectivity of $\Theta_{\alpha}(z,x_1,x_0)$ which therefore has a right inverse, too, also given by \eqref{2.7ib}. This fact then implies the following identities:
\begin{align}
\phi_{\alpha} (z,x,x_0)\theta_{\alpha} (\bar z,x,x_0)^*-
\theta_{\alpha} (z,x,x_0)\phi_{\alpha} (\bar z,x,x_0)^*&=0, \label{2.7j}
\\
\phi_{\alpha}' (z,x,x_0)\theta_{\alpha}' (\bar z,x,x_0)^*-
\theta_{\alpha}' (z,x,x_0)\phi_{\alpha}' (\bar z,x,x_0)^*&=0, \label{2.7k}
\\
\phi_{\alpha}' (z,x,x_0)\theta_{\alpha} (\bar z,x,x_0)^*-
\theta_{\alpha}' (z,x,x_0)\phi_{\alpha} (\bar z,x,x_0)^*&=I_{\cH}, \label{2.7l}
\\
\theta_{\alpha} (z,x,x_0)\phi_{\alpha}' (\bar z,x,x_0)^*-
\phi_{\alpha} (z,x,x_0)\theta_{\alpha}' (\bar z,x,x_0)^*&=I_{\cH}. \label{2.7m}
\end{align}

Having established the invertibility of $\Theta_\alpha(z,x_1,x_0)$ we can now show that
for any $x_1\in(a,b)$, any $\cH$-valued solution of $\tau y=zy$ may be expressed in terms of
$\theta_{\alpha}(z,\cdot,x_1)$ and $\phi_{\alpha}(z,\cdot,x_1)$, that is,
\begin{equation}
y(x)=\theta_\alpha(z,x,x_1)f+\phi_\alpha(z,x,x_1)g
\end{equation}
for appropriate vectors $f,g\in\cH$ or $\cB(\cH)$.

Next we establish a variation of constants formula.

\begin{lemma} \label{l2.8}
Suppose $F:(a,b)\to\cB(\cH)$ is a weakly measurable operator-valued function such that
$\|F(\cdot)\|_{\cB(\cH)} \in L^1_{\loc}((a,b); dx)$, assume that
$Y_0,Y_1\in\cB(\cH)$, and let $x_0 \in (a,b)$. Then the unique $\cB(\cH)$-valued solution
$Y(z,\cdot,x_0)$ of the initial value problem
\begin{equation}
\begin{cases} (\tau - z)Y = F, \\
\, Y(x_0)=Y_0, \; Y'(x_0)=Y_1, \end{cases}
\end{equation}
is given by $Y_h+Y_p$, where $Y_p$ is the particular solution of
$(\tau - z)Y = F$ $($in the sense of Definition \ref{d2.4}$)$ of the form
\begin{align}
\begin{split}
Y_p(x)& =\theta_\alpha(z,x,x_0)\int_{x_0}^x dx' \, \phi_\alpha(\bar z,x',x_0)^* F(x')   \\
& \quad  -\phi_\alpha(z,x,x_0)\int_{x_0}^x dx' \, \theta_\alpha(\bar z,x',x_0)^* F(x'),  \lb{2.19}
\end{split}
\end{align}
and $Y_h$ is the unique solution of the homogeneous initial value problem $($again
in the sense of Definition \ref{d2.4}$)$
\begin{equation}
\begin{cases} \tau Y=z Y, \\
Y(x_0)=Y_0, \; Y'(x_0)=Y_1. \end{cases}
\end{equation}
The analogous statement holds  when $F$ is replaced by
$f\in L^1_{\rm loc}((a,b);dx;\cH)$ and $Y_0, Y_1$ are replaced by $y_0, y_1\in\cH$.
\end{lemma}
\begin{proof}
This follows from a direct computation taking into account the identities \eqref{2.7j} and
\eqref{2.7l}.
\end{proof}

Finally we establish several versions of Green's formula (also called Lagrange's identity) which will be used frequently in the following.

\begin{lemma} \label{l2.9}
Let $(a,b)\subseteq\bbR$ be a finite or infinite interval and $[x_1,x_2]\subset(a,b)$. \\
$(i)$ Assume that $f,g\in W^{2,1}_{\rm loc}((a,b);dx;\cH)$. Then
\begin{equation}
\int_{x_1}^{x_2} dx \, [((\tau f)(x),g(x))_\cH-(f(x),(\tau g)(x))_\cH]
= W_{*}(f,g)(x_2)-W_{*}(f,g)(x_1).     \lb{2.52A}
\end{equation}
$(ii)$ Assume that $F:(a,b)\to\cB(\cH)$ is absolutely continuous, that $F'$ is again differentiable, and that $F''$ is weakly measurable. Also assume that $\|F''\|_\cH \in L^1_\loc((a,b);dx)$ and $g\in W^{2,1}_{\rm loc}((a,b);dx;\cH)$. Then
\begin{equation}
\int_{x_1}^{x_2} dx \, [(\tau F^*)^*(x)g(x)-F(x)(\tau g)(x)]
= (Fg'-F'g)(x_2)-(Fg'-F'g)(x_1).     \lb{2.52B}
\end{equation}
$(iii)$ Assume that $F,\,G:(a,b)\to\cB(\cH)$ are absolutely continuous operator-valued functions such that $F',\,G'$ are again differentiable and that $F''$, $G''$ are weakly measurable. In addition, suppose that $\|F''\|_\cH,\, \|G''\|_\cH \in L^1_\loc((a,b);dx)$. Then
\begin{equation}
\int_{x_1}^{x_2} dx \, [(\tau F^*)(x)^*G(x) - F(x) (\tau G)(x)] = (FG'-F'G)(x_2)-(FG'-F'G)(x_1).
\lb{2.53A}
\end{equation}
\end{lemma}
\begin{proof}
The product rule for scalar products
\begin{equation}
\f{d}{dx}(f(x),g(x))_\cH =(f(x),g'(x))_\cH+(f'(x),g(x))_\cH
\end{equation}
implies, as usual, the formula for integration by parts. Equation \eqref{2.52A} is then an immediate consequence of the latter and the fact that $V$ is self-adjoint so that $(Vf,g)_\cH=(f,Vg)_\cH$.

To prove \eqref{2.52B}, we first note that $g:(a,b)\to\cH$ is strongly continuous so that, by Lemma \ref{l2.1} the function $F''g$ is (strongly) measurable and integrable.
Lemma \ref{l2.1} then shows that also $Fg''$ and $FVg$ are measurable. Consequently, the integral on the left-hand side of \eqref{2.52B} is well-defined in the strong sense. The remainder of the proof relies again on a product rule. The product rule follows from the fact that each summand in
\begin{align}
\begin{split}
& \bigg\|F(x+\varepsilon)\bigg(\frac{g(x+\varepsilon)-g(x)}{\varepsilon}-g'(x)\bigg)\bigg\|
+ \|(F(x+\varepsilon)-F(x))g'(x)\| \\
& \quad + \bigg\|\bigg(\frac{F(x+\varepsilon)g(x)-F(x)g(x)}{\varepsilon}
- F'(x)g(x)\bigg)\bigg\| 
\end{split}
\end{align}
tends to zero as $\varepsilon \downarrow 0$, recalling that $x \in (a,b)$ is fixed.

Finally, to prove \eqref{2.53A}, we first note that $Gh:(a,b)\to\cH$ is strongly continuous for any $h\in\cH$. Again, Lemma \ref{l2.1} shows that $F''Gh$ is strongly measurable and integrable for any $h\in\cH$. The same applies to the terms $FG''h$ and $FVGh$. Consequently, the integral on the left-hand side of \eqref{2.53A} is well-defined in the strong sense. The stated equality \eqref{2.53A} now follows from an integration by parts as before.
\end{proof}

\begin{lemma} \label{l2.10}
Suppose that $y_0, y_1 \in \cH$ and either $x_0 \in (a,b)$ or $x_0$ is a regular endpoint of
$\tau$. Let $y(z,\cdot, x_0)$ be the unique solution of
\begin{equation}
\begin{cases} \tau y=zy, \\
y(x_0)=y_0, \; y'(x_0)=y_1. \end{cases}
\end{equation}
Then there is a constant $c_0 > 0$ and a constant $C(z,V)\leq 1$ depending only on $z$ and
$V$ such that
\begin{equation}
\int_{x_0}^x dx' \, \|y(x')\|_\cH^2 \geq c_0^2 (x-x_0)^3
\big\|(y_0,y_1)^\top\big\|_{\cH\oplus\cH}^2
\end{equation}
provided $0\leq x-x_0\leq C(z,V)$. A similar estimate holds for $x<x_0$.
\end{lemma}
\begin{proof}
Define $r(t)=y(t)-y_0-(t-x_0)y_1$. Then $-r''=(z-V)y$ so that the vector version of the variation of constants formula (Lemma \ref{l2.8}) treating $(z-V)y$ as the non-homogeneous term implies
\begin{equation}
r(x)=\int_{x_0}^x dx' \, (x'-x)[z-V(x')] y(x').
\end{equation}
Hence,
\begin{align}
\begin{split}
\|r(x)\|_{\cH} &\leq \sqrt2 (x-x_0) \big\|(y_0,y_1)^\top\big\|_{\cH\oplus\cH}
\int_{x_0}^x dx' \, \|z-V(x')\|_{\cB(\cH)}  \\
& \quad  +(x-x_0)\int_{x_0}^x dx' \, \|z-V(x')\|_{\cB(\cH)} \|r(x')\|_{\cH},
\end{split}
\end{align}
provided $|x-x_0|\leq1$. Gronwall's lemma then implies the estimate
\begin{equation}
\|r(x)\|_{\cH} \leq C \big\|(y_0,y_1)^\top\big\|_{\cH\oplus\cH} (x-x_0)
\int_{x_0}^x dx' \, \|z-V(x')\|_{\cB(\cH)}
\end{equation}
for an appropriate constant $C$ depending on $V-z$. Thus, using an integration by parts,
\begin{equation}
\int_{x_0}^x dx' \, \|r(x')\|_\cH^2 \leq \frac13 C^2 \big\|(y_0,y_1)^\top\big\|_{\cH\oplus\cH}^2
(x-x_0)^3 \bigg(\int_{x_0}^x dx' \, \|z-V(x')\|_{\cB(\cH)}\bigg)^2.   \lb{2.28a}
\end{equation}

On the other hand,
\begin{align}
& \int_{x_0}^x dx' \, \|y_0+(x'-x_0)y_1\|_\cH^2   \no \\
& \quad \geq (x-x_0)\|y_0\|_\cH^2-(x-x_0)^2\|y_0\|_\cH\|y_1\|_\cH
+ \frac13(x-x_0)^3 \|y_1\|_\cH^2   \no \\
& \quad  \geq 4c_0^2 (x-x_0)^3 (\|y_0\|_\cH^2+\|y_1\|_\cH^2)
\end{align}
for some constant $c_0>0$, provided $x-x_0$ is sufficiently
small (for instance, $c_0=1/10$ will do if $0\leq x-x_0\leq 1$). Combining this with
\eqref{2.28a} yields
\begin{align}
\begin{split}
\bigg(\int_{x_0}^x dx' \, \|y(x')\|_\cH^2 \bigg)^{1/2} &\geq (x-x_0)^{3/2}
\big\|(y_0,y_1)^\top\big\|_{\cH\oplus\cH}   \lb{2.30a} \\
& \quad \times \bigg[2c_0-C\int_{x_0}^x dx' \, \|z-V(x')\|_{\cB(\cH)}\bigg].
\end{split}
\end{align}
Finally, if $x$ is sufficiently close to $x_0$ in \eqref{2.30a}, the term inside the square brackets will be larger than $c_0$.
\end{proof}

\section{Weyl--Titchmarsh Theory}
\lb{s3}

In this section we develop Weyl--Titchmarsh theory for self-adjoint Schr\"odinger 
operators $H_{\alpha}$ in $L^2((a,b); dx; \cH)$ associated with the operator-valued differential expression $\tau =-(d^2/dx^2)+V(\cdot)$, assuming regularity of the 
left endpoint $a$ and the limit point case at the right endpoint $b$ (see  
Definition \ref{d3.6}). We prove the existence of Weyl--Titchmarsh solutions, introduce 
the corresponding Weyl--Titchmarsh $m$-function, and determine the structure of the Green's function of $H_{\alpha}$.

The broad outline of our approach in this section follows to a certain degree the 
path taken in the scalar case by Bennewitz \cite[Chs.\ 10, 11]{Be08}, 
Edmunds and Evans \cite[Sect.\ III.10]{EE89}, and Weidmann \cite[Sect.\ 8.4]{We80}. 
However, the operator-valued context also necessitates crucial deviations from 
the scalar approach as will become clear in the course of this section.  

We note that the boundary triple approach (see, e.g., \cite{DM91}, \cite{DM95} 
\cite{MN11}, \cite{MN11a}, \cite[Chs.\ 3, 4]{GG91} and the extensive literature cited 
therein) constitutes an alternative way to introduce 
operator-valued Weyl--Titchmarsh functions. However, we are not aware that this 
approach has been established for potentials $V$ satisfying our general 
Hypothesis \ref{h2.7}. Moreover, we intend to derive the existence of Weyl--Titchmarsh 
solutions from first principles and with minimal technical efforts.  

As before, $\cH$ denotes a separable Hilbert space and $(a,b)$ denotes a finite or infinite interval. One recalls that $L^2((a,b);dx;\cH)$ is separable (since $\cH$ is) 
and that
\begin{equation}
(f,g)_{L^2((a,b);dx;\cH)} =\int_a^b dx \, (f(x),g(x))_\cH, \quad f,g\in L^2((a,b);dx;\cH).
\end{equation}

Assuming Hypothesis \ref{h2.7} throughout this section, we are interested in
studying certain self-adjoint operators in $L^2((a,b);dx;\cH)$ associated with the 
operator-valued differential expression $\tau =-(d^2/dx^2)+V(\cdot)$. These will be suitable restrictions of the {\it maximal} operator $\oT_{\max}$ in $L^2((a,b);dx;\cH)$ defined by
\begin{align}
& \oT_{\max} f = \tau f,   \no \\
& f\in \dom(\oT_{\max})=\big\{g\in L^2((a,b);dx;\cH) \,\big|\, g\in W^{2,1}_{\rm loc}((a,b);dx;\cH); \\ 
& \hspace*{6.6cm} \tau g\in L^2((a,b);dx;\cH)\big\}.     \no 
\end{align}
We also introduce the operator $\dot \oT_{\min}$ in $L^2((a,b);dx;\cH)$ as the restriction of $\oT_{\max}$ to the domain
\begin{equation}
\dom(\dot \oT_{\min})=\{g\in\dom(\oT_{\max})\,|\,\supp (u) \, \text{is compact in} \, (a,b)\}.
\end{equation}
Finally, the {\it minimal} operator $\oT_{\min}$ in $L^2((a,b);dx;\cH)$ associated with $\tau$ is then defined as the closure of $\dot \oT_{\min}$,
\begin{equation}
\oT_{\min} = \ol{\dot \oT_{\min}}.
\end{equation}

Next, we intend to show that $\oT_{\max}$ is the adjoint of $\dot \oT_{\min}$ (and hence that of $\oT_{\min}$), implying, in particular, that $\oT_{\max}$ is closed. To this end, we first establish the following two preparatory lemmas for the case where $a$ and $b$ are both regular endpoints for $\tau$ in the sense of Definition \ref{d2.6}.

\begin{lemma} \label{l3.1}
In addition to Hypothesis \ref{h2.7} suppose that $a$ and $b$ are regular endpoints for $\tau$. Then
\begin{align} \label{3.5A}
\begin{split}
& \ker(\oT_{\max}-z I_{L^2((a,b);dx;\cH)})  \\
& \quad =\{[\theta_0(z,\cdot,a)f+\phi_0(z,\cdot,a)g]
\in L^2((a,b);dx;\cH) \,|\, f,g\in\cH\}
\end{split} 
\end{align}
is a closed subspace of $L^2((a,b);dx;\cH)$.
\end{lemma}
\begin{proof}
It is clear that the set on the right-hand side of \eqref{3.5A} is contained in
$\ker(\oT_{\max}-z I_{L^2((a,b);dx;\cH)})$. The existence and uniqueness result,
Theorem \ref{t2.3}, also
 establishes the converse inclusion. Thus, we only need to show that
 $\ker(\oT_{\max}-z I_{L^2((a,b);dx;\cH)})$ is a closed subspace of $L^2((a,b);dx;\cH)$ (one recalls that we did not yet establish that $\oT_{\max}$ is a closed operator).

Suppose that $\{u_n\}_{n \in \bbN} \subset \ker(\oT_{\max}-z I_{L^2((a,b);dx;\cH)})$ is a Cauchy sequence with respect to the topology in $L^2((a,b);dx;\cH)$. By Lemma \ref{l2.10} one has for some $\varepsilon > 0$,
\begin{align}
\begin{split}
\|u_n-u_m\|_{L^2((a,b);dx;\cH)}^2 & \geq \int_a^{a+\varepsilon}
dx \, \|u_n(x)-u_m(x)\|_\cH^2 \\
& \geq c_0^2 \varepsilon^3 \|(u_n(a)-u_m(a),u_n'(a)-u_m'(a))\|_{\cH\oplus\cH}^2.
\end{split}
\end{align}
This implies that both $\{u_n(a)\}_{n\in\bbN}$ and $\{u_n'(a)\}_{n\in\bbN}$ are Cauchy sequences in $\cH$ and hence convergent. Denoting the limits by $f$ and $g$, respectively, one concludes that
$u=[\theta_0(z,\cdot,a)f+\phi_0(z,\cdot,a)g] \in \ker(\oT_{\max}-z I_{L^2((a,b);dx;\cH)})$. Since
\begin{equation}
\|u_n-u\|_{L^2((a,b);dx;\cH)}\leq [2(b-a)]^{1/2} \big[C_1(z) \|u_n(a)-f\|_\cH
+ C_2(z) \|u_n'(a)-g\|_\cH\big],
\end{equation}
where
\begin{equation}
C_1(z) = \max_{x\in[a,b]} \|\theta_0(z,x,a)\|_{\cB(\cH)}, \quad
C_2(z) = \max_{x\in[a,b]} \|\phi_0(z,x,a)\|_{\cB(\cH)}, 
\end{equation}
the element $u$ is the strong limit if of the sequence $u_n$ in $L^2((a,b);dx;\cH)$ and hence
$\ker(\oT_{\max}-z I_{L^2((a,b);dx;\cH)})$ is closed.
\end{proof}

\begin{remark} \lb{r3.2}
If $\cH$ is finite-dimensional (e.g., in the scalar case, $\dim(\cH)=1$), then
$\ker(\oT_{\max}-z I_{L^2((a,b);dx;\cH)})$ is finite-dimensional and hence automatically closed.
\end{remark}

\begin{lemma} \label{l3.3}
In addition to Hypothesis \ref{h2.7} suppose that $a$ and $b$ are regular endpoints for $\tau$. Denote by $\oT_0$ the linear operator in $L^2((a,b);dx;\cH)$
defined by the restriction of
$\oT_{\max}$ to the space
\begin{equation}
\dom(\oT_0)=\{g\in\dom(\oT_{\max}) \,|\, g(a)=g(b)=g'(a)=g'(b)=0\}.
\lb{3.9a}
\end{equation}
Then
\begin{equation}
\ker(\oT_{\max})=[\ran(\oT_0)]^\perp,
\end{equation}
that is, the space of solutions $u$ of $\tau u=0$ coincides with the orthogonal complement of the collection of elements $\tau u_0$ satisfying $u_0\in \dom(\oT_0)$.
\end{lemma}
\begin{proof}
Suppose $u\in\ker(\oT_{\max})$ and $u_0\in \dom(\oT_0)$. Let $f_0=\oT_0u_0$. Then Green's formula \eqref{2.52A} yields $(f_0,u)_{L^2((a,b);dx;\cH)}=0$ so that
$\ran(\oT_0)\subseteq [\ker(\oT_{\max})]^\perp$.

Next, assume that $f_0\in [\ker(\oT_{\max})]^\perp$. Since $f_0$ is integrable, there is a solution $u_0$ of the initial value problem $\tau u_0=f_0$, $u_0(b)=u_0'(b)=0$.
If $u_1\in \ker(\oT_{\max})$, one has
\begin{equation}
0=(f_0,u_1)_{L^2((a,b);dx;\cH)}=-(u_0(a),u_1'(a))_\cH+(u_0'(a),u_1(a))_\cH,
\end{equation}
using Green's formula \eqref{2.52A} once more. Since one can choose $u_1$ so that
$u_1'(a)=0$ and $u_1(a)$ is an arbitrary vector in $\cH$, one necessarily concludes that  $u_0'(a) = 0$. Similarly, choosing $u_1(a)=0$ and $u_1'(a)$ arbitrarily shows that $u_0(a)=0$. Hence $u_0\in \dom(\oT_0)$ and $f_0\in \ran(\oT_0)$.

We have now shown that $\ran(\oT_0) = [\ker(\oT_{\max})]^\perp$. Taking orthogonal complements and recalling from Lemma \ref{l3.1} that $\ker(\oT_{\max})$ is closed, concludes the proof of Lemma \ref{l3.3}.
\end{proof}

\begin{theorem} \label {t3.4}
Assume Hypothesis \ref{h2.7}. Then the operator $\dot \oT_{\min}$ is densely defined. Moreover, $\oT_{\max}$ is the adjoint of $\dot \oT_{\min}$,
\begin{equation}
\oT_{\max} = (\dot \oT_{\min})^*.   \lb{3.12a}
\end{equation}
In particular, $\oT_{\max}$ is closed. In addition, $\dot \oT_{\min}$ is symmetric and
$\oT_{\max}^*$ is the closure of $\dot \oT_{\min}$, that is,
\begin{equation}
\oT_{\max}^* = \ol{\dot \oT_{\min}} = \oT_{\min}.    \lb{3.13a}
\end{equation}
\end{theorem}
\begin{proof}
Suppose $f_1$ is perpendicular to $\dom(\dot \oT_{\min})$ and let $u_1$ be a solution
of $\tau u_1=f_1$. Let $[\tilde a, \tilde b]$ be a compact interval contained in $(a,b)$ and introduce the operators $\wti {\oT_{\max}}$ and $\wti {\dot \oT_{\min}}$ associated
with that interval and acting in the Hilbert space
$\wti{L^2((a,b);dx;\cH)} =L^2\big(\big(\tilde a,\tilde b\big);dx;\cH\big)$ with inner product $(\cdot,\cdot)_{\wti{L^2((a,b);dx;\cH)}}$. We extend any function
$u_0\in \dom\Big(\wti {\dot \oT_{\min}}\Big)$ by zero outside the interval $[\tilde a, \tilde b]$ to get an element of $\dom(\dot \oT_{\min})$, also denoted by $u_0$. Similarly, we consider the restriction of $f_1$ to $[\tilde a, \tilde b]$, and for simplicity, also denote it by $f_1$. Thus, setting $f_0=\tau u_0$, we get via Green's formula \eqref{2.52A}
\begin{equation}
0=(u_0,f_1)_{L^2((a,b);dx;\cH)}=(u_0,f_1)_{\wti{L^2((a,b);dx;\cH)}}
=(f_0,u_1)_{\wti{L^2((a,b);dx;\cH)}}.
\end{equation}
Lemma \ref{l3.3} then implies that $u_1\in\ker\Big(\wti {\oT_{\max}}\Big)$ and
hence that $f_1$ is zero almost everywhere in $[\tilde a, \tilde b]$. Since we may
choose $\tilde a$ arbitrarily close to $a$, and $\tilde b$ arbitrarily close to $b$, we get $f_1=0$ a.e., proving that $\dot \oT_{\min}$ is densely defined.

To show that $\oT_{\max}$ is the adjoint of $\dot \oT_{\min}$ (and hence a closed operator), we first recall that the domain of $(\dot \oT_{\min})^*$ is given by
\begin{align}
& \dom\big(\big(\dot \oT_{\min}\big)^*\big)=\{u\in L^2((a,b);dx;\cH) \,|\, \text{there exists 
$u^*\in L^2((a,b);dx;\cH)$, such}  \no \\ 
& \hspace*{.7cm} \text{that for all } u_0\in \dom(\dot \oT_{\min}),   
 (\dot \oT_{\min}u_0,u)_{L^2((a,b);dx;\cH)}
=(u_0,u^*)_{L^2((a,b);dx;\cH)}\}.  
\end{align}
The inclusion $\dom(\oT_{\max})\subseteq \dom(\big(\dot \oT_{\min})^*)$ then follows immediately from Green's formula \eqref{2.52A} because we can choose $u^*$ to
be $\tau u$ whenever $u\in\dom(\oT_{\max})$.

For proving the reverse inclusion, let $u\in \dom((\dot \oT_{\min})^*)$, note that
$u^*=(\dot \oT_{\min})^*u$ is locally integrable, and let $h$ be a solution of the differential equation $\tau h=u^*$. As a consequence of Green's formula \eqref{2.52A} one
obtains that
\begin{align}
\begin{split}
\int_a^b dx \, (\tau v,u-h)_\cH &=(\dot \oT_{\min}v,u)_{L^2((a,b);dx;\cH)}
- \int_a^b dx \, (\tau v,h)_\cH    \\
&=(v,u^*)_{L^2((a,b);dx;\cH)}-\int_a^b dx \, (v,\tau h)_\cH = 0,
\end{split}
\end{align}
whenever $v\in \dom(\dot \oT_{\min})$.  Thus, the restriction of $u-h$ to any
interval $[\tilde a, \tilde b]\supseteq \supp(v)$ is orthogonal to
$\ran\Big(\wti {\dot \oT_{\min}}\Big)$ and hence lies in
$\ker\Big(\wti {\oT_{\max}}\Big)$.
This shows that $u$ and $u'$ are locally absolutely continuous and that 
$\tau u=u^*\in L^2((a,b);dx;\cH)$, that is, $u\in\dom(\oT_{\max})$.

Since
\begin{equation}
\dot \oT_{\min} \subseteq \oT_{\max} = (\dot \oT_{\min})^*,
\end{equation}
$\dot \oT_{\min}$ is symmetric in $L^2((a,b);dx;\cH)$. Hence $\oT_{\max}^*$ is a
restriction of $\oT_{\max}$ and thus an extension of $\dot \oT_{\min}$. Finally, \eqref{3.13a} is an immediate consequence of \eqref{3.12a}.
\end{proof}

Lemmas \ref{l3.1}, \ref{l3.3}, and Theorem \ref{t3.4}, under additional hypotheses on $V$ (typically involving continuity assumptions) are of course well-known and go back to 
Rofe-Beketov \cite{Ro69}, \cite{Ro69a} (see also \cite[Sect.\ 3.4]{GG91}, \cite[Ch.\ 5]{RK05}). 
 
\begin{remark} \lb{3.5}
In the special case where $a$ and $b$ are regular endpoints for $\tau$, the operator $H_0$ introduced in \eqref{3.9a} coincides with the minimal
operator $\oT_{\min}$.
\end{remark}

Using the dominated convergence theorem and Green's formula \eqref{2.52A} one can show that $\lim_{x\to a}W_*(u,v)(x)$ and $\lim_{x\to b}W_*(u,v)(x)$ both exist whenever
$u,v\in\dom(\oT_{\max})$. We will denote these limits by $W_*(u,v)(a)$ and $W_*(u,v)(b)$, respectively. Thus Green's formula also holds for $x_1=a$ and $x_2=b$ if $u$ and $v$ are in $\dom(\oT_{\max})$, that is,
\begin{equation}\label{3.17A}
(\oT_{\max}u,v)_{L^2((a,b);dx;\cH)}-(u,\oT_{\max}v)_{L^2((a,b);dx;\cH)}
= W_*(u,v)(b) - W_*(u,v)(a).
\end{equation}
This relation and the fact that $\oT_{\min}=\oT_{\max}^*$ is a restriction of $\oT_{\max}$ show that
\begin{align}
\begin{split}
& \dom(\oT_{\min})=\{u\in\dom(\oT_{\max}) \,|\, W_*(u,v)(b)=W_*(u,v)(a)=0   \\
& \hspace*{6.2cm} \text{ for all } v\in \dom(\oT_{\max})\}.     \label{3.18A}
\end{split}
\end{align}

\begin{definition} \lb{d3.6}
Assume Hypothesis \ref{h2.7}.
Then the endpoint $a$ (resp., $b$) is said to be of {\it limit-point type for $\tau$} if
$W_*(u,v)(a)=0$ (resp., $W_*(u,v)(b)=0$) for all $u,v\in\dom(\oT_{\max})$.
\end{definition}

By using the term ``limit-point type'' one recognizes Weyl's contribution to the
subject in his celebrated paper \cite{We10}.

Next, we introduce the subspaces
\begin{equation}
\cD_{z}=\{u\in\dom(\oT_{\max}) \,|\, \oT_{\max}u=z u\}, \quad z \in \bbC.
\end{equation}
For $z\in\bbC\backslash\bbR$, $\cD_{z}$ represent the deficiency subspaces of
$\oT_{\min}$. Von Neumann's theory of extensions of symmetric operators implies that
\begin{equation} \label{3.20A}
\dom(\oT_{\max})=\dom(\oT_{\min}) \dotplus \cD_i \dotplus \cD_{-i}
\end{equation}
where $\dotplus$ indicates the direct (but not necessarily orthogonal direct) sum.

\begin{lemma} \label{l3.7}
Assume Hypothesis \ref{h2.7}. Suppose $a$ is a regular endpoint for $\tau$,
let $f_1 \in \cH$, $f_2\in \cH$. Then there are elements
$u\in\dom(\oT_{\max})$ such that $u(a)=f_1$, $u'(a)=f_2$, and $u$ vanishes on $[c,b)$ for some $c \in (a,b)$. The analogous statements hold with the roles of $a$ and $b$ interchanged.
\end{lemma}
\begin{proof}
Let $h=[\theta_0(0,\cdot,a)g_1+\phi_0(0,\cdot,a)g_2]\chi_{[a,c]}$, where
$g_1\in \cH$, $g_2 \in \cH$, and $c\in(a,b)$ are as yet undetermined.
Then $h\in L^2((a,b);dx;\cH)$. Solving the initial value problem $\tau u=h$,
$u(c)=u'(c)=0$, implies that $u\in\dom (\oT_{\max})$ and that $u$ is zero
on $[c,b)$. Moreover, Green's formula \eqref{2.52B} shows that
\begin{equation}
\int_a^c dx' \, \theta_0(0,x',a)^*h(x') = \int_a^c dx' \, \theta_0(0,x',a)^*(-u''+Vu) = u'(a)
\end{equation}
and
\begin{equation}
\int_a^c dx' \, \phi_0(0,x',a)^*h(x') = \int_a^c dx' \, \phi_0(0,x',a)^*(-u''+Vu) = -u(a).
\end{equation}
We want to choose $g_1$ and $g_2$ so that $u(a)=f_1$ and $u'(a)=f_2$, that is,
$A_c(g_1,g_2)^\top=(f_2,-f_1)^\top$, where $A_c:\cH\oplus\cH\to \cH\oplus\cH$ is given by
\begin{equation}
A_c=\begin{pmatrix} \int_a^c dx' \, \theta_0(0,x',a)^*\theta_0(0,x',a)
&\int_a^c dx' \, \theta_0(0,x',a)^*\phi_0(0,x',a)    \\[2mm]
 \int_a^c dx' \, \phi_0(0,x',a)^*\theta_0(0,x',a)
 &\int_a^c dx' \, \phi_0(0,x',a)^*\phi_0(0,x',a)
 \end{pmatrix}.
\end{equation}
Hence the proof will be complete if we can show that $A_c$
is invertible for a proper choice of $c$. Let $F=(g_1,g_2)^\top\in\cH\oplus\cH$. Since
\begin{equation}
(F,A_c F)_{\cH\oplus \cH}
= \int_a^c dx' \, \|\theta_0(0,x',a)g_1 + \phi_0(0,x',a)g_2\|_\cH^2,
\end{equation}
and since
$\theta_0(0,x',a)g_1 + \phi_0(0,x',a)g_2 =0$ only if $g_1=g_2=0$, it follows that $A_c$ is positive definite and hence injective. To show that $A_c$ is also surjective we will prove
that $(F,A_c F)_{\cH \oplus \cH} \geq \gamma \|F\|_\cH^2$ for some constant
$\gamma > 0$ since this implies that zero cannot be in the approximate point spectrum of $A_c$ (we recall that the spectrum and approximate point spectrum coincide for self-adjoint operators and refer for additional comments to the paragraph preceding Lemma \ref{l3.12}).

By Lemma \ref{l2.10},
\begin{equation}
(F,A_c F)_{\cH \oplus \cH}=\int_a^c dx' \, \|\theta_0(0,x',a)g_1
+ \phi_0(0,x',a)g_2\|_\cH^2
\geq c_0^2 (c-a)^3 \|F\|_{\cH\oplus\cH}^2
\end{equation}
provided $c-a$ is sufficiently small. Thus, $\gamma$ can be chosen
as $c_0^2 (c-a)^3$.
\end{proof}

We now set out to determine the self-adjoint restrictions of $\oT_{\max}$ assuming
that $a$ is a regular endpoint for $\tau$ and $b$ is of limit-point type for $\tau$. To this
end we first briefly  recall the concept
of a Hermitian relation. For more information the reader may consult, for instance,
\cite[Appendix\ A]{RK05}.

A subset $\cM$ of $\cH\oplus\cH$ is called a {\it Hermitian relation} in the Hilbert space
$\cH$ if it has the following two properties:
\begin{enumerate}
  \item If $(f_1,f_2)$ and $(g_1,g_2)$ are in $\cM$, then $(f_1,g_2)_\cH=(f_2,g_1)_\cH$.
  \item If $(f_1,f_2)\in\cH\oplus\cH$ and $(f_1,g_2)_\cH=(f_2,g_1)_\cH$ for all
  $(g_1,g_2)\in \cM$, then $(f_1,f_2)\in \cM$.
\end{enumerate}

Thus, a Hermitian relation is a linear subspace of $\cH\oplus\cH$ and one can show
that $\cM=\wti \cM$ if $\cM$ and $\wti \cM$ are Hermitian relations such that
$\cM\subseteq \wti \cM$. Moreover, the following lemma holds:

\begin{lemma} \lb{l3.8}
The maps $\pi_\pm: \cM\to\cH: (f_1,f_2)\mapsto f_\pm=f_2 \pm i f_1$ are linear
bijections and $U=\pi_-\circ \pi_+^{-1}:\cH\to\cH$ is unitary\footnote{We note that $U$ is called the Cayley transform of $\cM$.}.
\end{lemma}
\begin{proof}
It is clear that $\pi_\pm$ are linear. If $(f_1,f_2)\in \cM$, a straightforward calculation yields
\begin{equation}\label{3.26A}
\|f_\pm\|_\cH^2=\|f_1\|_\cH^2+\|f_2\|_\cH^2
\end{equation}
and so proves injectivity of $\pi_\pm$ and that $U$ is a partial isometry. The proof will be finished when we show that $\pi_\pm$ are also surjective.

We begin by showing that the range of $\pi_+$ is dense in $\cH$. To do so assume
that $g\in\cH$ is orthogonal to $f_2+i f_1$, that is,
$0=(g,f_2 + i f_1)_\cH=(g,f_2)_\cH - (ig,f_1)_\cH$ for all $(f_1,f_2)\in \cM$. This implies
that $(g,i g)\in \cM$. Then $\pi_-(g,ig)=0$ and, using \eqref{3.26A},
we have $g=0$. Now let $f_+ \in \cH$. Then there is a sequence $(f_{1,n},f_{2,n})\in \cM$, $n\in\bbN$, such that $f_{2,n} + i f_{1,n}$ converges to $f_+$. Thus $f_{2,n} + i f_{1,n}$
is Cauchy in $\cH$ and \eqref{3.26A} entails that $f_{1,n}$ and $f_{2,n}$,
$n\in\bbN$, are separately Cauchy and hence convergent in $\cH$. Denote the limit
of $(f_{1,n},f_{2,n})$ as $n \to\infty$ by $(f_1,f_2)$. In view of the continuity of scalar products one finds that
\begin{equation}
(f_1,g_2)_\cH=\lim_{n\to\infty}(f_{1,n},g_2)_\cH=\lim_{n\to\infty}(f_{2,n},g_1)_\cH
=(f_2,g_1)_\cH, \quad (g_1,g_2)\in \cM.
\end{equation}
This implies that $(f_1,f_2)\in \cM$ and $f_+=f_2 + i f_1 \in\ran(\pi_+)$. Surjectivity
of $\pi_-$ is shown in the same manner.
\end{proof}

Next, suppose that $\alpha$ is a (bounded or unbounded) self-adjoint operator in $\cH$. Then
\begin{equation}
\cM_\alpha=\{(f_1,f_2)\in\cH\oplus\cH \,|\, \sin(\alpha)f_2 + \cos(\alpha)f_1 = 0\}
\end{equation}
is a Hermitian relation. This follows since $\sin(\alpha)f_2 + \cos(\alpha)f_1= 0$ if and
only if there is an $h\in\cH$ such that $f_1 = - \sin(\alpha)h$ and $f_2 = \cos(\alpha)h$.
In fact, $h = \cos(\alpha)f_2 - \sin(\alpha)f_1$, if $(f_1,f_2)\in \cM_\alpha$ is given.

We now use the theory of Hermitian relations to characterize all self-adjoint restrictions
of $\oT_{\max}$ under the following set of assumptions:

\begin{hypothesis} \lb{h3.9}
In addition to Hypothesis \ref{h2.7} suppose that $a$ is a regular
endpoint for $\tau$ and $b$ is of limit-point type for $\tau$.
\end{hypothesis}

\begin{theorem} \lb{t3.10}
Assume Hypothesis \ref{h3.9}. If $\oT$ is a self-adjoint restriction of
$\oT_{\max}$, then there is a bounded and self-adjoint operator $\alpha\in\cB(\cH)$
such that
\begin{equation}
\dom(\oT)=\{u\in\dom(\oT_{\max}) \,|\, \sin(\alpha)u'(a) + \cos(\alpha)u(a)=0\}.  \lb{3.29A}
\end{equation}
Conversely, for every $\alpha \in \cB(\cH)$, \eqref{3.29A} gives rise to a self-adjoint restriction of $\oT_{\max}$ in $L^2((a,b);dx;\cH)$.
\end{theorem}
\begin{proof}
Suppose $\oT=\oT^*\subseteq \oT_{\max}$ and define
\begin{equation}
\cM=\{(f_1,f_2)\in\cH\oplus\cH \,|\,\text{there exists } u\in\dom(\oT)
\text{ such that } f=u(a), f'=u'(a)\}.
\end{equation}
We show first that $\cM$ is a Hermitian relation: For $(f_1,f_2),(g_1,g_2)\in \cM$ let
$u,v\in\dom(\oT)$ be such that $u(a)=f_1$, $u'(a)=f_2$, $v(a)=g_1$, and
$v'(a)=g_2$. Since
$\oT$ is self-adjoint one infers from Green's formula \eqref{2.52A} that
\begin{align}
\begin{split} 
& 0=(\oT u,v)_{L^2((a,b);dx;\cH)}-(u,\oT v)_{L^2((a,b);dx;\cH)}   \\ 
& \quad =-W_*(u,v)(a)=(u'(a),v(a))_\cH-(u(a),v'(a))_\cH.
\end{split} 
\end{align}
Next assume $(f_1,f_2)\in\cH\oplus\cH$ and that $(f_1,v'(a))_\cH=(f_2,v(a))_\cH$ for
all $v\in\dom(\oT)$. By Lemma \ref{l3.7} there is a $u\in\dom(\oT_{\max})$ with initial values $(f_1,f_2)$ and hence,
\begin{align}
\begin{split}
& (\oT_{\max}u,v)_{L^2((a,b);dx;\cH)}-(u,\oT v)_{L^2((a,b);dx;\cH)}   \\
& \quad =-W_*(u,v)(a)=(f_2,v(a))_\cH - (f_1,v'(a))_\cH=0.
\end{split} 
\end{align}
This implies that $u\in\dom(\oT^*)=\dom(\oT)$ (with $\oT^*u=\oT_{\max}u$) and hence that $(f_1,f_2)\in \cM$. Thus $\cM$ is indeed a Hermitian relation. Let $U$ be its Cayley transform, $\{F_U(t)\}_{t \in [0,2\pi]}$ the family of strongly
right-continuous spectral projections associated with $U$, implying\footnote{We employ the standard slight abuse of notation where $F_U(t) = F_U ([0,t))$,
$t \in [0,2\pi]$, and use the normalization
$\slim_{\varepsilon\downarrow 0} F_U(-\varepsilon) = 0$,
$F_U (2\pi) = \slim_{\varepsilon\downarrow 0} F_U(2\pi + \varepsilon) = I_{\cH}$.},
\begin{equation}
(f,Ug)_\cH=\int_{[0,2\pi]} \e^{it}d(f,F_U (t)g)_\cH,\quad F(0)=0,
\end{equation}
and $\alpha$ the bounded self-adjoint operator defined by
\begin{equation}
(f,\alpha g)_\cH=\frac12\int_{[0,2\pi]} t \, d(f,F_U (t)g)_\cH.
\end{equation}
Since $U$ is the Cayley transform of $\cM$, we have $U(f_2 + i f_1)=f_2 - i f_1$, or  equivalently, $(U-I_{\cH})f_2 + i (U+I_{\cH})f_1 = 0$. Since $U=\e^{2i\alpha}$, the latter relation implies that $\sin(\alpha)f_2 + \cos(\alpha) f_1 = 0$. Thus,
$\cM \subseteq \cM_\alpha$, implying (as shown in the paragraph preceding Lemma
\ref{l3.8}), that $\cM=\cM_\alpha$. Thus the first part of Theorem \ref{t3.10} follows.

For the converse part, assume
$\alpha = \alpha^* \in \cB(\cH)$ is given, and let $\oT$ denote the restriction of
$\oT_{\max}$ to those functions satisfying $\sin(\alpha)u'(a)+\cos(\alpha)u(a)=0$,
that is, $u\in\dom(\oT)$ if and only if
$(u(a),u'(a))\in \cM_\alpha$. Therefore, if $u,v\in\dom(\oT)$, then
$W_*(u,v)(a)=W_*(u,v)(b)=0$ so that
$(\oT u,v)_{L^2((a,b);dx;\cH)}=(u,\oT v)_{L^2((a,b);dx;\cH)}$, implying
$\dom(\oT)\subseteq \dom(\oT^*)$. To show the opposite inclusion one first notes that
$\dom(\oT^*)\subseteq\dom(\oT_{\max})$ since $\dom(\oT_{\max}^*)\subseteq\dom(\oT)$.
Now assume that $u\in\dom(\oT^*)$ and $v\in\dom(\oT)$. Then $\oT^*u=\oT_{\max}u$ so
that $W_*(u,v)(a)=0$ for all $(v(a),v'(a))\in \cM_\alpha$. This implies that
$(u(a),u'(a))\in \cM_\alpha$, that is, $\dom(\oT^*)\subseteq \dom(\oT)$.
\end{proof}

Henceforth, under the assumptions of Theorem \ref{t3.10}, we denote the operator $\oT$ in $L^2((a,b);dx;\cH)$ associated with the boundary condition induced by $\alpha = \alpha^* \in \cB(\cH)$, that is, the restriction of $\oT_{\max}$ to the set
\begin{equation}
\dom(H_{\alpha})=\{u\in\dom(\oT_{\max}) \,|\, \sin(\alpha)u'(a)+\cos(\alpha)u(a)=0\}
\end{equation}
by $H_{\alpha}$. For a discussion of boundary conditions at infinity, see, for instance, 
\cite{MN11}, \cite{Mo09}, and \cite{RK85}.

Our next goal is to construct the square integrable solutions $Y(z,\cdot) \in\cB(\cH)$
of $\tau Y=zY$, $z\in\bbC\backslash\bbR$, the $\cB(\cH)$-valued Weyl--Titchmarsh solutions, under the assumptions that $a$ is a regular endpoint for $\tau$ and $b$ is of limit-point type for $\tau$.

For ease of notation, we denote in the following the resolvent
of $H_{\alpha}$ by $R_{z,\alpha}$, that is,
$R_{z,\alpha}=(H_{\alpha} - z I_{L^2((a,b);dx;\cH)})^{-1}$.

One recalls that the graph of $H_{\alpha}$, given by
\begin{equation}
\Gamma=\{(f,H_{\alpha}f) \in L^2((a,b);dx;\cH) \oplus L^2((a,b);dx;\cH) \,|\, f\in\dom(H_{\alpha})\},
\end{equation}
is a Hilbert subspace of $L^2((a,b);dx;\cH)\oplus L^2((a,b);dx;\cH)$. Equivalently, one
can consider $\dom(H_{\alpha})$ as a Hilbert space with scalar product
\begin{align}
\begin{split}
(f,g)_\Gamma=\int_a^b dx \, (f(x),g(x))_\cH
+ \int_a^b dx \, ((H_{\alpha}f)(x), (H_{\alpha}g)(x))_\cH,&  \\
f, g \in \dom(H_{\alpha}),&
\end{split}
\end{align}
and the corresponding norm $\|f\|_\Gamma=(f,f)_\Gamma^{1/2}$,
$f \in \dom(H_{\alpha})$.
Given a compact interval $J\subset [a,b)$ we know that $\dom(H_{\alpha})$ is contained in the Banach space $C^1(J;\cH)$ of continuously differentiable functions
on $J$ with values in $\cH$ and norm given by
$\|f\|_J=\sup_{x \in J} \|f(x)\|_{\cH}+\sup_{x \in J} \|f'(x)\|_{\cH}$. In fact the following
lemma holds.

\begin{lemma} \label{l3.11}
Assume Hypothesis \ref{h3.9} and suppose that $\alpha \in \cB(\cH)$ is self-adjoint.
For each compact interval $J \subset [a,b)$ there is a constant $C_J$ such that
$\|y\|_J \leq C_J \|y\|_\Gamma$ for every $y\in\dom(H_{\alpha})$.
\end{lemma}
\begin{proof}
Suppose $\{y_n\}_{n\in\bbN} \subset \dom(H_{\alpha})$ is a sequence converging
to $y \in \dom(H_{\alpha})$ with respect to the norm $\|\cdot\|_\Gamma$ and that
$y_n|_J$ converges in $C^1(J;\cH)$ to $\tilde y$ as $n\to\infty$. It follows that
\begin{equation}
\|y_n-y\|_{L^2((a,b);dx;\cH)} + \|y_n-y\|_{L^2(J;dx;\cH)} \underset{n\to\infty}{\longrightarrow} 0.
\end{equation}
On account of the uniform convergence in $C^1(J;\cH)$ one also concludes that
$\|y_n-\tilde y\|_{L^2(J;dx;\cH)} \to 0$ as $n\to\infty$. Thus, $y|_J=\tilde y$ so that the restriction map $y\mapsto y|_J$ defined on $\dom(H_{\alpha})$ is closed and hence
bounded by the closed graph theorem.
\end{proof}

We recall that a point $\lambda \in \bbC$ is said to be in the {\it approximate point spectrum} of a closed operator $T\in\cB(\cH)$ if there is a sequence
$\{x_n\}_{n\in\bbN}\subset \cH$ such that $\|x_n\|_{\cH}=1$, $n \in \bbN$, and
$\lim_{n\to\infty}\|(T-\lambda I_{\cH})x_n\|_{\cH} = 0$. If $\lambda$ is an eigenvalue, then it is, of course, in the approximate point spectrum. $\lambda$ is also in the approximate point spectrum, if $T-\lambda I_{\cH}$ is injective and its image is dense in $\cH$ but not closed, a fact that can be seen as follows: In this case $(T-\lambda I_{\cH})^{-1}$ is a densely defined unbounded operator, that is, there is a sequence $f_n$ such that
$\|f_n\|_{\cH}=1$ and $\|(T-\lambda I_{\cH})^{-1}f_n\|_{\cH}>n$, $n\in\bbN$.  This is equivalent to the existence of a sequence $\{y_n\}_{n\in\bbN} \subset \cH$ (namely
$y_n=(T-\lambda)^{-1}f_n/\|(T-\lambda)^{-1}f_n\|$, $n\in\bbN$) such that
$\|y_n\|_{\cH}=1$ and $\|(T-\lambda I_{\cH})y_n\|_{\cH}<1/n$, $n\in\bbN$, so that
$\lambda$ is in the approximate point spectrum. If $T$ has no residual spectrum, in particular, if $T$ is self-adjoint, its spectrum coincides with its approximate point spectrum.

\begin{lemma} \label{l3.12}
Suppose $\alpha\in\cB(\cH)$ is self-adjoint. If $c_j \in\bbC$, $j=1,2$, with
$c_1/c_2 \in \bbC\backslash\bbR$, then $0 \in \rho(c_1 \sin(\alpha) + c_2 \cos(\alpha))$.
\end{lemma}
\begin{proof}
Let $A=\sin(\alpha)$, $B=\cos(\alpha)$, and assume that $c_j \in\bbC$, $j=1,2$, with
$c_1/c_2 \in \bbC\backslash\bbR$. The spectral theorem implies that the spectra
of $A$ and $B$ are contained in $[-1,1]$ and that the spectra of $A^2$ and $B^2$ are contained in $[0,1]$. By way of contradiction, assume that $0$ is in the approximate point spectrum of $c_1 A+c_2 B$. Then there is a sequence $\{x_n\}_{n\in\bbN} \subset \cH$ such that $\|x_n\|_{\cH}=1$, $n\in\bbN$, and $\lim_{n\to\infty}\|(c_1 A + c_2 B)x_n\|_{\cH} = 0$. Accordingly, also $\|(c_1^2A^2+c_1c_2AB)x_n\|_{\cH} \to 0$ and
$\|(c_1c_2BA+c_2^2 B^2)x_n\|_{\cH} \to 0$ as $n\to\infty$. Hence,
$(c_1^2A^2-c_2^2B^2)x_n = (c_1^2+c_2^2)A^2x_n-c_2^2x_n$ tends to zero as
$n\to\infty$, so that $c_2^2/(c_1^2+c_2^2)$ is in the approximate point spectrum of $A^2$. This implies that $c_1/c_2$ is real, a contradiction. Thus, $0$ is not in the approximate point spectrum of $c_1A+c_2B$. Hence, for $0$ to be in the spectrum of $c_1A+c_2B$ 
would require that its image not be dense in $\cH$, that is, that $\ker(\ol{c_1}A+\ol{c_2}B)=\ker((c_1A+c_2B)^*)=\ran(c_1A+c_2B)^\perp \supsetneq \{0\}$. But this is impossible as we have just shown.
\end{proof}

Fix $c\in(a,b)$ and $z \in \rho(H_{\alpha})$. For any $f_0\in\cH$ let $f=f_0\chi_{[a,c]}\in L^2((a,b);dx;\cH)$ and $u(f_0,z,\cdot)=R_{z,\alpha}f \in\dom(H_{\alpha})$. By the variation of constants formula,
\begin{align}
\begin{split}
u(f_0,z,x) &=\theta_{\alpha}(z,x,a)\bigg(g(z)+\int_x^c dx' \, \phi_{\alpha}(\ol z,x',a)^* f_0\bigg)\\
& \quad  +\phi_{\alpha}(z,x,a)\bigg(h(z)-\int_x^c dx' \, \theta_{\alpha}(\ol z,x',a)^* f_0\bigg)
 \end{split}
\end{align}
for suitable vectors $g(z) \in \cH$, $h(z) \in \cH$. Since
$u(f_0,z,\cdot)\in\dom(H_{\alpha})$, one infers that
\begin{equation} \label{3.40A}
g(z)=-\int_a^c dx' \, \phi_{\alpha}(\ol z,x',a)^* f_0, \quad z \in \rho(H_{\alpha}),
\end{equation}
and that
\begin{equation} \lb{3.40B}
h(z)=\cos(\alpha)u'(f_0,z,a) - \sin(\alpha)u(f_0,z,a)
+ \int_a^c dx' \, \theta_{\alpha}(\ol z,x',a)^* f_0, \quad z \in \rho(H_{\alpha}).
\end{equation}

\begin{lemma} \lb{l3.13}
Assume Hypothesis \ref{h3.9} and suppose that $\alpha \in \cB(\cH)$ is self-adjoint. In addition, choose $c \in (a,b)$ and introduce $g(\cdot)$ and $h(\cdot)$ as in \eqref{3.40A} and \eqref{3.40B}. Then the maps
\begin{equation}
C_{1,\alpha}(c,z):\begin{cases} \cH\to\cH, \\
f_0\mapsto g(z), \end{cases} \quad
C_{2,\alpha}(c,z): \begin{cases} \cH\to\cH, \\
f_0\mapsto h(z), \end{cases}  \quad z \in \rho(H_{\alpha}),
\end{equation}
are linear and bounded. Moreover, $C_{1,\alpha}(c,\cdot)$ is entire and
$C_{2,\alpha}(c,\cdot)$ is analytic on $\rho(H_{\alpha})$. In addition,
$C_{1,\alpha}(c,z)$ is boundedly invertible if $z\in \bbC\backslash\bbR$ and $c$
is chosen appropriately.
\end{lemma}
\begin{proof}
According to equation \eqref{3.40A} one has
\begin{equation}
C_{1,\alpha}(c,z)=-\int_a^c dx' \, \phi_{\alpha}(\ol z,x',a)^*.
\end{equation}
By Corollary \ref{c2.5}\,$(iii)$, $C_{1,\alpha}(c,\cdot)$ is entire.

Next, one observes that $\rho(H_{\alpha}) \ni z\mapsto u(f_0,z,x)=(R_{z,\alpha}f)(x)$
is analytic and its derivative at $z_0$ is given by $(R_{z_0,\alpha}^2f)(x)$. This follows
from Lemma \ref{l3.11} and the first resolvent identity since
\begin{align}
& \left\|\frac{(R_{z,\alpha}f)(x)-(R_{z_0,\alpha}f)(x)}{z-z_0}
- (R_{z_0,\alpha}^2f)(x)\right\|_\cH
\leq \left\|\frac{R_{z,\alpha}f-R_{z_0,\alpha}f}{z-z_0}-R_{z_0,\alpha}^2f\right\|_J   \no \\
& \quad \leq C_J
\left\|\frac{R_{z,\alpha}f-R_{z_0,\alpha}f}{z-z_0}-R_{z_0,\alpha}^2f\right\|_\Gamma
\leq C_J \left\|(R_{z,\alpha}-R_{z_0,\alpha})R_{z_0,\alpha}f\right\|_\Gamma,
\end{align}
as long as $x \in J$, with $J \subset (a,b)$ a compact interval, noting in addition that
\begin{equation}
H_{\alpha}(R_{z,\alpha}-R_{z_0,\alpha})=zR_{z,\alpha}-z_0R_{z_0,\alpha}.
\end{equation}
Similarly, $z\mapsto u'(f_0,z,x)=(R_{z,\alpha}f)'(x)$ is analytic, proving that
$C_{2,\alpha}(c,\cdot)$ is analytic on $\rho(H_{\alpha})$.

It remains to show the bounded invertibility of $C_{1,\alpha}(c,z)$ for
$z\in\bbC\backslash\bbR$ and appropriate $c\in (a,b)$. In order for the expression
\begin{equation}
\frac{\tan(\mu)}{\mu}=\frac{1-\cos(2\mu)}{\mu\sin(2\mu)}, \quad \mu \in\bbC,
\end{equation}
to be real-valued it is necessary that $\mu$ be either real or purely imaginary. Hence, using Lemma \ref{l3.12}, one finds that the operator
\begin{align}
\begin{split}
S &=\sin(\alpha)\frac{\sin(k(c-a))}{k}+\cos(\alpha)\frac{\cos(k(c-a))-1}{k^2}\\
 &=\int_a^c dx' \,\bigg[\sin(\alpha)\cos(k(x'-a)) - \cos(\alpha)\frac{\sin(k(x'-a))}{k}\bigg]
 \end{split}
\end{align}
is boundedly invertible unless $k^2 \in \bbR$. A proof similar to that of Lemma \ref{l2.10}
then shows that
\begin{equation}
\big\|C_{1,\alpha}(c,k^2)-S\big\|_{\cB(\cH)}
\end{equation}
is arbitrarily small for $c-a$ is sufficiently small. This proves that $C_{1,\alpha}(c,z)$ is boundedly invertible if $z\in\bbC\backslash\bbR$ and $c$ is chosen appropriately.
\end{proof}

Using the bounded invertibility of $C_{1,\alpha}(c,z)$ we now define
\begin{equation}
\psi_{\alpha}(z,x)=\theta_{\alpha}(z,x,a)
+ \phi_{\alpha}(z,x,a)C_{2,\alpha}(c,z)C_{1,\alpha}(c,z)^{-1},
\quad z \in \bbC\backslash\bbR, \; x \in [a,b),     \lb{3.49A}
\end{equation}
still assuming Hypothesis \ref{h3.9} and $\alpha = \alpha^* \in \cB(\cH)$. By Lemma \ref{l3.13}, $\psi_{\alpha}(\cdot,x)$ is analytic on
$z \in \bbC\backslash\bbR$ for fixed $x \in [a,b]$.

Since $\psi_{\alpha}(z,\cdot) f_0$ is the solution of the initial value problem
\begin{equation}
\tau y =z y, \quad y(c)=u(f_0,z,c), \; y'(c)=u'(f_0,z,c), \quad z \in \bbC\backslash\bbR,
\end{equation}
the function $\psi_{\alpha}(z,x)C_{1,\alpha}(z,c)f_0$ equals
$u(f_0,z,x)$ for $x\geq c$, and thus is square integrable for every choice of $f_0\in\cH$.
In particular, choosing $c \in (a,b)$ such that $C_{1,\alpha}(z,c)^{-1} \in \cB(\cH)$, one infers that
\begin{equation}
\int_a^b dx \, \|\psi_{\alpha}(z,x) f\|_{\cH}^2 < \infty, \quad
f \in \cH, \; z \in \bbC\backslash\bbR.
\end{equation}

Every $\cH$-valued solution of $\tau y=z y$ may be written as
\begin{equation}
y=\theta_\alpha(z,\cdot,a)f_{\alpha,a} + \phi_\alpha(z,\cdot,a)g_{\alpha,a},
\end{equation}
with
\begin{equation}
f_{\alpha,a}=(\cos\alpha)y(a)+(\sin\alpha)y'(a), \quad
g_{\alpha,a}=-(\sin\alpha)y(a)+(\cos\alpha)y'(a).
\end{equation}
Hence we can define the maps
\begin{align}
& \mc C_{1,\alpha,z}:\begin{cases} \cD_z\to\cH, \\
\theta_\alpha(z,\cdot,a) f_{\alpha,a}
+ \phi_\alpha(z,\cdot,a) g_{\alpha,a} \mapsto f_{\alpha,a}, \end{cases} \\
& \mc C_{2,\alpha,z}: \begin{cases} \cD_z\to\cH, \\
\theta_\alpha(z,\cdot,a) f_{\alpha,a} + \phi_\alpha(z,\cdot,a) g_{\alpha,a} \mapsto g_{\alpha,a}. \end{cases}
\end{align}

\begin{lemma} \label{l3.14}
Assume Hypothesis \ref{h3.9}, suppose that $\alpha \in \cB(\cH)$ is self-adjoint, and
let $z\in\bbC\backslash\bbR$. Then the operators
$\mc C_{1,\alpha,z}$ and $\mc C_{2,\alpha,z}$ are linear bijections and hence 
\begin{equation}
\mc C_{1,\alpha,z}, \, \mc C_{1,\alpha,z}^{-1}, \, \mc C_{2,\alpha,z}, \, 
\mc C_{2,\alpha,z}^{-1} \in \cB(\cH).   \lb{3.57}
\end{equation}
\end{lemma}
\begin{proof}
It is clear that $\mc C_{1,\alpha,z}$ and $\mc C_{2,\alpha,z}$ are linear. Given
$f\in\cH$ one concludes that $u=\psi_{\alpha}(z,\cdot)f$ and
$v=\psi_{\alpha+\pi/2}(z,\cdot)f$
are in $\cD_z$ and $\mc C_{1,\alpha,z}u=\mc C_{2,\alpha,z}v=f$. This proves surjectivity of $\mc C_{1,\alpha,z}$ and $\mc C_{2,\alpha,z}$.

Next, let $u=\theta_\alpha f+\phi_\alpha g\in\cD_z$ and $f=0$ or $g=0$. Then
$W_*(u,u)(a)=0$. Moreover, since $b$ is of limit-point type for $\tau$, $W_*(u,u)(b)=0$. Hence, by \eqref{3.17A},
\begin{align}
\begin{split} 
& 0 = (\oT_{\max}u,u)_{L^2((a,b);dx;\cH)}-(u,\oT_{\max}u)_{L^2((a,b);dx;\cH)}  \\
& \quad = (zu,u)_{L^2((a,b);dx;\cH)}-(u,zu)_{L^2((a,b);dx;\cH)}
=(\ol z -z)\|u\|_{L^2((a,b);dx;\cH)}^2, 
\end{split} 
\end{align}
implying $u=0$ and injectivity of $\mc C_{1,\alpha,z}$ and $\mc C_{2,\alpha,z}$. 
Since for any invertible operator $T$ in $\cH$ one has that $T^{-1}$ is closed if 
and only if $T$ is (cf.\ \cite[Sect.\ III.5.2]{Ka80}), the closed graph theorem 
(see, \cite[Sect.\ III.5.4]{Ka80}) yields \eqref{3.57}. 
\end{proof}

At this point we are finally in the position to define the 
Weyl--Titchmarsh $m$-function for $z\in\bbC\backslash\bbR$ by setting
\begin{equation} \label{3.57A}
m_{\alpha}(z)=\mc C_{2,\alpha,z}\mc C_{1,\alpha,z}^{-1}, \quad
z\in\bbC\backslash\bbR.
\end{equation}

\begin{theorem} \label{t3.15}
Assume Hypothesis \ref{h3.9} and suppose that $\alpha \in \cB(\cH)$ is self-adjoint. Then 
\begin{equation}
m_{\alpha}(z) \in \cB(\cH), \quad z\in\bbC\backslash\bbR,   \lb{3.57B}
\end{equation}
and $m_{\alpha}(\cdot)$ is analytic on $\bbC\backslash\bbR$. Moreover,
\begin{equation}
m_{\alpha}(z)=m_{\alpha}(\ol z)^*, \quad z \in \bbC\backslash\bbR.    \lb{3.59A}
\end{equation}
\end{theorem}
\begin{proof}
The boundedness relation \eqref{3.57B} immediately follows from \eqref{3.57} and \eqref{3.57A}. To prove analyticity we first show that 
$m_{\alpha}(z)=C_{2,\alpha}(c,z)C_{1,\alpha}^{-1}(c,z)$ where $C_{1,\alpha}$, 
$C_{2,\alpha}$ and $c$ are as in Lemma \ref{l3.13}. To this end let $h$ be an arbitrary element of $\cH$. Then
\begin{align}
\mc C_{2,\alpha,z}\mc C_{1,\alpha,z}^{-1}h
& = \mc C_{2,\alpha,z} \psi_{\alpha}(z, \cdot) h  \no \\
& =\mc C_{2,\alpha,z}(\theta_\alpha(z,\cdot,a)h
+ \phi_\alpha(z,\cdot,a)C_{2,\alpha}(c,z)C_{1,\alpha}(c,z)^{-1}h)    \no \\
& =C_{2,\alpha}(c,z)C_{1,\alpha}(c,z)^{-1}h    \lb{3.60A}
\end{align}
establishing the claimed identity. The analyticity of $m_{\alpha}$ on
$\bbC\backslash\bbR$ now follows from Lemma \ref{l3.13}.

To prove \eqref{3.59A} one first observes that \eqref{2.7f}--\eqref{2.7i} yield
\begin{equation}
W(\psi_{\alpha}(\ol z,\cdot)^*,\psi_{\alpha}(z,\cdot))(x)
= m_{\alpha}(z)-m_{\alpha}(\ol z)^*.
\end{equation}
Fixing arbitrary $f,g\in\cH$, then yields
\begin{equation}
(f,(m_{\alpha}(z)-m_{\alpha}(\ol z)^*)g)_\cH =
W(\psi_{\alpha}(\ol z,\cdot)^*f,\psi_{\alpha}(z,\cdot)g)(x)
\underset{x \uparrow b}{\longrightarrow} 0,
\end{equation}
since both $\psi_{\alpha}(\ol z,\cdot)f$ and $\psi_{\alpha}(z,\cdot)g$ are in
$\dom(\oT_{\max})$ and since $b$ is of limit-point-type for $\tau$.
\end{proof}

As a consequence of \eqref{3.60A}, the $\cB(\cH)$-valued function $\psi_{\alpha}(z,\cdot)$
in \eqref{3.49A} can be rewritten in the form
\begin{equation} \label{3.58A}
\psi_{\alpha}(z,x)=\theta_{\alpha}(z,x,a)+\phi_{\alpha}(z,x,a)m_{\alpha}(z),
\quad z \in \bbC\backslash\bbR, \; x \in [a,b).
\end{equation}
In particular, this implies that $\psi_{\alpha}(z,\cdot)$ is independent of
the choice of the parameter $c \in (a,b)$ in \eqref{3.49A}.
Following the tradition in the scalar case ($\dim(\cH) = 1$), we will call
$\psi_{\alpha}(z,\cdot)$ the {\it Weyl--Titchmarsh} solution associated
with $\tau Y = z Y$.

We remark that, given a function $u\in\cD_z$, the operator $m_{0}(z)$ assigns the Neumann boundary data $u'(a)$ to the Dirichlet boundary data
$u(a)$, that is, $m_{0}(z)$ is the ($z$-dependent)
Dirichlet-to-Neumann map.

With the aid of the Weyl--Titchmarsh solutions we can now give a detailed description of the resolvent
$R_{z,\alpha} = (H_{\alpha} - z I_{L^2((a,b);dx;\cH)})^{-1}$ of $H_{\alpha}$.

\begin{theorem}\label{t3.16}
Assume Hypothesis \ref{h3.9} and suppose that $\alpha \in \cB(\cH)$ is self-adjoint. Then the resolvent of $H_{\alpha}$ is an integral operator of the type
\begin{align}
\begin{split}
\big((H_{\alpha} - z I_{L^2((a,b);dx;\cH)})^{-1} u\big)(x)
= \int_a^b dx' \, G_{\alpha}(z,x,x')u(x'),& \\
u \in L^2((a,b);dx;\cH), \; z \in \rho(H_{\alpha}), \; x \in [a,b),&
\end{split}
\end{align}
with the $\cB(\cH)$-valued Green's function $G_{\alpha}(z,\cdot,\cdot) $ given by
\begin{equation} \label{3.63A}
G_{\alpha}(z,x,x') = \begin{cases}
\phi_{\alpha}(z,x,a) \psi_{\alpha}(\ol{z},x')^*, & a\leq x \leq x'<b, \\
\psi_{\alpha}(z,x) \phi_{\alpha}(\ol{z},x',a)^*, & a\leq x' \leq x<b,
\end{cases}  \quad z\in\bbC\backslash\bbR.
\end{equation}
\end{theorem}
\begin{proof}
First assume that $u\in L^2((a,b);dx;\cH)$ is compactly supported and let
\begin{equation}
v(x)=\psi_{\alpha}(z,x)\int_a^x \phi_{\alpha}(\ol{z},x',a)^* u(x') dx' +\phi_{\alpha}(z,x,a) \int_x^b \psi_{\alpha}(\ol{z},x')^* u(x')dx'.
\end{equation}
We need to show that $v=R_{z,\alpha}u$. To this end one notes that both $v$ and $v'$ are in $W^{(1,1)}_{\rm loc}((a,b),dx;\cH)$. Near the endpoints $v$ is a multiple of either $\phi_{\alpha}(z,\cdot,a)$ or $\psi_{\alpha}(z,\cdot)$. Hence it satisfies the boundary condition at $a$ and is square integrable. Differentiating once more shows that $\tau v=u$ so that 
$v\in L^2((a,b);dx;\cH)$ and $v=R_{z,\alpha}u$. The fact that compactly supported functions are dense in $L^2((a,b);dx;\cH)$ completes the proof.
\end{proof}

One recalls from Definition \ref{dA.4} that a nonconstant function 
$N:\bb C_+\to\cB(\cH)$ is called a (bounded) operator-valued Herglotz function, if 
$z\mapsto (u,N(z)u)_{\cH}$ is analytic and 
has a non-negative imaginary part for all $u\in\cH$.

\begin{theorem} \label{t3.17}
Assume Hypothesis \ref{h3.9} and suppose that $\alpha \in \cB(\cH)$ and
$\beta \in \cB(\cH)$ are self-adjoint.
Then the $\cB(\cH)$-valued function $m_{\alpha}(\cdot)$ 
is an operator-valued Herglotz function and explicitly determined by the Green's function
for $H_{\alpha}$ as follows,
\begin{align}
& m_{\alpha}(z) = \big(-\sin(\alpha), \cos(\alpha)\big)
\begin{pmatrix} G_{\alpha}(z,a,a) & G_{\alpha,x'} (z,a,a) \\
G_{\alpha, x} (z,a,a) & G_{\alpha, x,x'} (z,a,a) \end{pmatrix}
\begin{pmatrix} -\sin(\alpha) \\ \cos(\alpha) \end{pmatrix},    \no \\
& \hspace*{9.5cm}   z\in\bbC\backslash\bbR,     \label{2.19a}
\end{align}
where we denoted
\begin{align}
G_{\alpha,x} (z,a,a) &= \slim_{\substack{x' \to a \\ a<x<x'}}
\frac{\partial}{\partial x} G_{\alpha} (z,x,x'),   \no   \\
G_{\alpha,x'} (z,a,a) &= \slim_{\substack{x' \to a \\ a<x<x'}}
\frac{\partial}{\partial x'} G_{\alpha} (z,x,x'),     \\
G_{\alpha,x,x'} (z,a,a) &= \slim_{\substack{x' \to a \\ a<x<x'}}
\frac{\partial}{\partial x} \frac{\partial}{\partial x'} G_{\alpha} (z,x,x')    \no
\end{align}
$($the strong limits referring to the strong operator topology in $\cH$$)$.
In addition, $m_{\alpha}(\cdot)$ extends analytically to the resolvent set
of $H_{\alpha}$.

Moreover, $m_{\alpha}(\cdot)$ and $m_{\beta}(\cdot) $ are related by the
following linear fractional transformation,
\begin{equation}\label{3.67A}
m_{\beta}=(C+Dm_{\alpha})(A+Bm_{\alpha})^{-1},
\end{equation}
where
\begin{equation}
\begin{pmatrix}A&B\\ C&D\end{pmatrix}
= \begin{pmatrix}\cos(\beta) & \sin(\beta) \\ -\sin(\beta) & \cos(\beta) \end{pmatrix}
\begin{pmatrix}\cos(\alpha) & -\sin(\alpha) \\ \sin(\alpha) & \cos(\alpha) \end{pmatrix}.
\end{equation}
\end{theorem}
\begin{proof}
Pick $z\in\bbC\backslash\bbR$ throughout this proof.  
We begin by establishing the validity of the linear fractional transformation. Let $\psi$ be any $\cH$-valued square integrable solution of $\tau\psi=z\psi$. Since
\begin{equation}
\psi(x)=\theta_{\alpha}(z,\cdot,a)f+\phi_{\alpha}(z,\cdot,a)g = \theta_{\beta}(z,\cdot,a)u+\phi_{\beta}(z,\cdot,a)v
\end{equation}
for appropriate $f,g,u,v\in\cH$, one gets
\begin{equation}
\begin{pmatrix}u\\ v\end{pmatrix}=\begin{pmatrix}A&B\\ C&D\end{pmatrix}\begin{pmatrix}f\\ g\end{pmatrix}.
\end{equation}
Since $v=m_{\beta}u$, $g=m_{\alpha}f$, and since $A+B m_{\alpha}(z)=\mc C_{1,\beta,z}\mc C_{1,\alpha,z}^{-1}$ is invertible, one obtains \eqref{3.67A}.

In view of this relationship between $m$-operators for different boundary conditions we prove the first part of the theorem first for a specific boundary condition, namely
$\alpha_0=\frac\pi2 I_{\cH}$ so that $\sin(\alpha_0)=I_{\cH}$ and $\cos(\alpha_0)=0$.
Then, for every $\varepsilon > 0$ there is a $\delta > 0$ such that
$\|\theta_{\pi/2}(a,z,x)\|_{\cB(\cH)}$ and $\|\phi_{\pi/2}(a,z,x)-I_{\cH}\|_{\cB(\cH)}$ are smaller than
$\varepsilon$ provided $x-a<\delta$. Next, for any fixed $u_0\in\cH$ let
$u_{\delta}=u_0 \chi_{[a,a+\delta]}/\delta^{1/2}$. Using Theorems \ref{t3.15} and \ref{t3.16}, one obtains
\begin{align} \label{3.71A}
& (u_{\delta}, R_{z,\pi/2} u_{\delta})_{L^2((a,b);dx;\cH)}   \no \\
& \quad = \int_a^{a + \delta} dx \, \bigg\{
\big(u_{\delta}(x),\theta_{\pi/2}(z,x,a)\int_a^x dx' \,
\phi_{\pi/2}(\ol z,x',a)^* u_{\delta}(x')\big)_{\cH}  \no \\
 & \qquad + \big(u_{\delta}(x),\phi_{\pi/2}(z,x,a)\int_x^b dx' \,
 \theta_{\pi/2}(\ol z,x',a)^* u_{\delta}(x')\big)_{\cH}
 \no \\
& \qquad +\big(u_{\delta}(x),[\phi_{\pi/2}(z,x,a) - I_{\cH}]
m_{\pi/2}(z) \int_a^b dx' \,
\phi_{\pi/2}(\ol z,x',a)^* u(x')\big)_{\cH}  \no \\
& \qquad + \big(u(x),m_{\pi/2}(z) \int_a^b dx' \,
(\phi_{\pi/2}(\ol z,x',a)^*-I) u_{\delta}(x')\big)_{\cH}\bigg\}
\no \\
& \qquad +(u_0,m_{\pi/2}(z)u_0)_\cH.
\end{align}
Hence,
\begin{align}
\begin{split}
& |(u_0,m_{\pi/2}(z)u_0)_\cH-(u_{\delta}, R_{z,\pi/2} u_{\delta})_{L^2((a,b);dx;\cH)}|  \\
& \quad \leq (\varepsilon (1+2\|m_{\pi/2}(z)\|)
+\varepsilon^2(1+\|m_{\pi/2}(z)\|))\|u_0\|^2.
\end{split}
\end{align}
Since $\delta$ goes to zero with $\varepsilon$ one gets
\begin{align}
\begin{split}
& \Im\big((u_0,m_{\pi/2}(z)u_0)_\cH\big)
= \lim_{\delta\downarrow 0}
\Im \big((u_{\delta}, R_{z,\pi/2} u_{\delta})_{L^2((a,b);dx;\cH)}\big)   \\
& \quad = \Im (z) \lim_{\delta\downarrow 0} \int_\bb R
\frac{d(u_{\delta}, E_{H_{\pi/2}}(-\infty,t] u_{\delta})_{L^2((a,b);dx;\cH)}}{|t-z|^2} 
\geq 0,
\end{split}
\end{align}
where $E_{H_{\pi/2}}(\cdot)$ denotes the strongly right-continuous family of spectral projections  associated with $H_{\pi/2}$. Since we already showed that $m_{\pi/2}$ is analytic away from the real axis, it follows that it is an operator-valued Herglotz function.

It remains to show that $m_{\beta}$ possesses the Herglotz property for general 
$\beta$. Using \eqref{3.67A} for $\alpha=\pi/2$ and setting
$v_0=(A+Bm_{\pi/2})^{-1}u_0$ for an arbitrary element $u_0$ of $\cH$ one finds
\begin{align}
2i\Im\big((u_0,m_{\beta}u_0)_\cH\big)
&= (u_0,m_{\beta}u_0)_\cH-(m_{\beta}u_0,u_0)_\cH  \no \\
&= (v_0,m_{\pi/2}v_0)_\cH-(m_{\pi/2} v_0,v_0)_\cH  \no \\
&= 2i\Im \big((v_0,m_{\pi/2}v_0)_\cH\big) \geq 0,
\end{align}
proving that $m_{\beta}$ is Herglotz.

Finally, \eqref{2.19a} follows by a simple calculation.
\end{proof}

We also mention that $G_{\alpha}(\cdot,x,x)$ is a bounded Herglotz operator 
in $\cH$ for each $x\in (a,b)$, as is clear from \eqref{2.7j}, \eqref{3.58A}, 
\eqref{3.63A}, and the Herglotz property of $m_{\alpha}$. 

\begin{remark} \lb{r3.19}
The Weyl--Titchmarsh theory established in this section is modeled after 
right half-lines $(a,b) = (0, \infty)$. Of course precisely the analogous theory applies 
to left half-lines $(-\infty,0)$. Given the two half-line results, one then establishes 
the full-line result on $\bbR$ in the usual fashion with $x=0$ a reference point and a 
$2 \times 2$ block operator formalism as in well-known the scalar or matrix-valued 
cases; we omit further details at this point.  
\end{remark} 

\appendix
\section{Basic Facts on Operator-Valued Herglotz Functions} \lb{sA}
\setcounter{theorem}{0}
\setcounter{equation}{0}

In this appendix we review some basic facts on (bounded) operator-valued Herglotz functions, applicable to $m_{\alpha}$ and $G_{\alpha}(\cdot,x,x)$, $x\in (a,b)$, 
discussed in the bulk of this paper.

In the remainder of this appendix, let $\cH$ be a separable, complex Hilbert space
with inner product denoted by $(\cdot,\cdot)_{\cH}$.

\begin{definition}\label{dA.4}
The map $M: \bbC_+ \rightarrow \cB(\cH)$ is called a bounded
operator-valued Herglotz function in $\cH$ (in short, a bounded Herglotz operator in
$\cH$) if $M$ is analytic on $\bbC_+$ and $\Im (M(z))\geq 0$ for all $z\in \bbC_+$.
\end{definition}

Here we follow the standard notation
\begin{equation} \lb{A.37}
\Im (M) = (M-M^*)/(2i),\quad \Re (M) = (M+M^*)/2, \quad M \in \cB(\cH).
\end{equation}

Note that $M$ is a bounded Herglotz operator if and only if the scalar-valued functions 
$(u,Mu)_\cH$ are Herglotz for all $u\in\cH$. 

As in the scalar case one usually extends $M$ to $\bbC_-$ by
reflection, that is, by defining
\begin{equation}
M(z)=M(\overline z)^*, \quad z\in \bbC_-.   \lb{A.36}
\end{equation}
Hence $M$ is analytic on $\bbC\backslash\bbR$, but $M\big|_{\bbC_-}$
and $M\big|_{\bbC_+}$, in general, are not analytic
continuations of each other.

Of course, one can also consider unbounded operator-valued 
Herglotz functions, but they will not be used in this paper. 

In contrast to the scalar case, one cannot generally expect strict
inequality in $\Im(M(\cdot))\geq 0$. However, the kernel of $\Im(M(\cdot))$
has simple properties:

\begin{lemma} \lb{lA.5}
Let $M(\cdot)$ be a bounded operator-valued Herglotz function in $\cH$.
Then the kernel $\cH_0 = \ker(\Im(M(z)))$ is independent of $z\in\bbC_+$.
Consequently, upon decomposing $\cH = \cH_0 \oplus \cH_1$,
$\cH_1 = \cH_0^\bot$, $\Im(M(\cdot))$ takes on the form
\begin{equation}
\Im(M(z))= \begin{pmatrix} 0 & 0 \\ 0 & N_1(z) \end{pmatrix},
\quad z \in \bbC_+,     \lb{A.38}
\end{equation}
where $N_1(\cdot) \in \cB(\cH_1)$ satisfies
\begin{equation}
N_1(z) > 0, \quad z\in\bbC_+.    \lb{A.39}
\end{equation}
\end{lemma}
For a proof of Lemma \ref{lA.5} see, for instance,
\cite[Proposition\ 1.2\,$(ii)$]{DM97} (alternatively, the
proof of \cite[Lemma\ 5.3]{GT00} in the matrix-valued context extends to the present
infinite-dimensional situation).

Next we recall the definition of a bounded operator-valued measure (see, also
\cite[p.\ 319]{Be68}, \cite{MM04}, \cite{PR67}):

\begin{definition} \lb{dA.6}
Let $\cH$ be a separable, complex Hilbert space.
A map $\Sigma:\mathfrak{B}(\bbR) \to\cB(\cH)$, with $\mathfrak{B}(\bbR)$ the
Borel $\sigma$-algebra on $\bbR$, is called a {\it bounded, nonnegative,
operator-valued measure} if the following conditions $(i)$ and $(ii)$ hold: \\
$(i)$ $\Sigma (\emptyset) =0$ and $0 \leq \Sigma(B) \in \cB(\cH)$ for all
$B \in \mathfrak{B}(\bbR)$. \\
$(ii)$ $\Sigma(\cdot)$ is strongly countably additive (i.e., with respect to the
strong operator  \hspace*{5mm} topology in $\cH$), that is,
\begin{align}
& \Sigma(B) = \slim_{N\to \infty} \sum_{j=1}^N \Sigma(B_j)   \lb{A.40} \\
& \quad \text{whenever } \, B=\bigcup_{j\in\bbN} B_j, \, \text{ with } \,
B_k\cap B_{\ell} = \emptyset \, \text{ for } \, k \neq \ell, \;
B_k \in \mathfrak{B}(\bbR), \; k, \ell \in \bbN.    \no
\end{align}
In addition, $\Sigma(\cdot)$ is called an {\it $($operator-valued\,$)$ spectral
measure} (or an {\it orthogonal operator-valued measure}) if the following
condition $(iii)$ holds: \\
$(iii)$ $\Sigma(\cdot)$ is projection-valued (i.e., $\Sigma(B)^2 = \Sigma(B)$,
$B \in \mathfrak{B}(\bbR)$) and $\Sigma(\bbR) = I_{\cH}$. \\ 
$(iv)$ Let $f \in \cH$ and $B \in \mathfrak{B}(\bbR)$. Then the vector-valued 
measure $\Sigma(\cdot) f$ has {\it finite variation on $B$}, denoted by 
$V(\Sigma f;B)$, if 
\begin{equation}
V(\Sigma f; B) = \sup\bigg\{\sum_{j=1}^N \|\Sigma(B_j)f\|_{\cH} \bigg\} < \infty,
\end{equation}
where the supremum is taken over all finite sequences $\{B_j\}_{1\leq j \leq N}$ 
of pairwise disjoint subsets on $\bbR$ with $B_j \subseteq B$, $1 \leq j \leq N$. 
In particular, $\Sigma(\cdot) f$ has {\it finite total variation} if 
$V(\Sigma f;\bbR) < \infty$. 
\end{definition}

We recall that due to monotonicity considerations (cf., \eqref{A.52}), taking the limit 
in the strong operator topology in \eqref{A.40} is equivalent to taking the limit with 
respect to the weak operator topology in $\cH$. 

We also note that integrals of the type \eqref{A.41a}--\eqref{A.42a} below are now 
taken with respect to an operator-valued measure, as opposed to the Bochner integrals 
we used in the bulk of this paper, Sections \ref{s2} and \ref{s3}.   

For relevant material in connection with the following result we refer the reader, for instance, to \cite{AL95}, \cite{AN75}, \cite{AN76},
\cite{BT92}, \cite[Sect.\ VI.5,]{Be68}, \cite[Sect.\ I.4]{Br71}, \cite{Bu97}, \cite{Ca76},
\cite{De62}, \cite{DM91}--\cite{DM97}, \cite[Sects.\ XIII.5--XIII.7]{DS88},
\cite{HS98}, \cite{KO77}, \cite{KO78}, \cite{Ma62}, \cite{MM02}, \cite{MM04}, 
\cite[Ch.\ VI]{Na68}, \cite{Na74}, \cite{Na77}, \cite{NA75}, \cite{Sh71}, \cite{Ts92},
\cite[Sects.\ 8--10]{We87}. 

\begin{theorem}
\rm {(\cite{AN76}, \cite[Sect.\ I.4]{Br71}, \cite{Sh71}.)} \lb{tA.7}
Let $M$ be a bounded operator-valued Herglotz function in $\cH$.
Then the following assertions hold: \\
$(i)$ For each $f \in \cH$, $(f,M(\cdot) f)_{\cH}$ is a $($scalar$)$ Herglotz function. \\
$(ii)$ Suppose that $\{e_j\}_{j\in\bbN}$ is a complete orthonormal system in $\cH$ 
and that for some subset of $\bbR$ having positive Lebesgue measure, and for all 
$j\in\bbN$, $(e_j,M(\cdot) e_j)_{\cH}$ has zero normal limits. Then $M\equiv 0$. \\ 
$(iii)$ There exists a bounded, nonnegative $\cB(\cH)$-valued measure
$\Omega$ on $\bbR$ such that the Nevanlinna representation
\begin{align}
& M(z) = C + D z + \int_{\bbR}
\f{d\Omega (\lambda)}{1+\lambda^2} \, \frac{1 + \lambda z}{\lambda -z}    \lb{A.41a} \\
& \qquad \; = C + D z + \int_{\bbR} d\Omega (\lambda ) \bigg[\f{1}{\lambda-z}
- \f{\lambda}{1+\lambda ^2}\bigg],
\quad z\in\bbC_+,    \lb{A.42} \\
& \Omega((-\infty, \lambda]) = \slim_{\varepsilon \downarrow 0}
\int_{-\infty}^{\lambda + \varepsilon} \f{d \Omega (t)}{1 + t^2},  \quad 
\lambda \in \bbR,  \\
& \Omega(\bbR) = \Im(M(i)) 
= \int_{\bbR} \f{d\Omega(\lambda)}{1+\lambda^2} \in \cB(\cH),   \lb{A.42a} \\
& C=\Re(M(i)),\quad D=\slim_{\eta\uparrow \infty} \,
\frac{1}{i\eta}M(i\eta) \geq 0,      \lb{A.42b}
\end{align}
holds in the strong sense in $\cH$.  \\
$(iv)$ Let $\lambda _1,\lambda_2\in\bbR$, $\lambda_1<\lambda_2$. Then the
Stieltjes inversion formula for $\Omega $ reads
\begin{equation}\lb{A.43}
\Omega ((\lambda_1,\lambda_2]) f =\pi^{-1} \slim_{\delta\downarrow 0}
\slim_{\varepsilon\downarrow 0}
\int^{\lambda_2 + \delta}_{\lambda_1 + \delta}d\lambda \,
\Im (M(\lambda+i\varepsilon)) f, \quad f \in \cH.
\end{equation}
$(v)$ Any isolated poles of $M$ are simple and located on the
real axis, the residues at poles being nonpositive bounded operators in $\cB(\cH)$.  \\
$(vi)$ For all $\lambda \in \bbR$,
\begin{align}
& \slim_{\varepsilon \downarrow 0} \, \varepsilon
\Re(M(\lambda +i\varepsilon ))=0,    \lb{A.45} \\
& \, \Omega (\{\lambda \}) = \slim_{\varepsilon \downarrow 0} \,
\varepsilon \Im (M(\lambda + i \varepsilon )) 
=- i \slim_{\varepsilon \downarrow 0} \,
\varepsilon M(\lambda +i\varepsilon).     \lb{A.46}
\end{align} 
$(vii)$ If in addition $M(z) \in \cB_{\infty} (\cH)$, $z \in \bbC_+$, then the measure
$\Omega$ in \eqref{A.41a} is countably additive with respect to the $\cB(\cH)$-norm
and the Nevanlinna representation \eqref{A.41a}, \eqref{A.42} and the
Stieltjes inversion formula \eqref{A.43} as well as \eqref{A.45}, \eqref{A.46} hold 
with the limits taken with respect to the $\|\cdot\|_{\cB(\cH)}$-norm. \\
$(viii)$ Let $f \in \cH$ and assume in addition that $\Omega(\cdot) f$ is of finite total 
variation. Then for a.e.\ $\lambda \in \bbR$, the normal limits $M(\lambda + i0) f$ 
exist in the strong sense and 
\begin{equation}
\slim_{\varepsilon \downarrow 0} M(\lambda +i\varepsilon) f 
= M(\lambda +i 0) f = H(\Omega(\cdot) f) (\lambda) + i \pi \Omega'(\lambda) f,
\end{equation}
where $H(\Omega(\cdot) f)$ denotes the $\cH$-valued Hilbert transform
\begin{equation}
H(\Omega(\cdot) f) (\lambda) = \text{p.v.}\int_{- \infty}^{\infty} d \Omega (t) f \, 
\f{1}{t - \lambda} 
= \slim_{\delta \downarrow 0} \int_{|t-\lambda|\geq \delta} d \Omega (t) f \, 
\f{1}{t - \lambda}.
\end{equation}
\end{theorem}
\noindent {\it Sketch of proof.} 
Item $(i)$ is clear and it implies items $(ii)$ together with the fact that 
$\sum_{j\in\bbN} 2^{-j} (e_j, \Omega(\cdot) e_j)_{\cH}$ represents a (scalar) control 
measure for $\Omega(\cdot)$. 

That equations \eqref{A.41a}--\eqref{A.42b} hold in the strong sense in $\cH$ and 
the validity of the Stieltjes inversion formula \eqref{A.43} were proved by Allen and 
Narcowich \cite{AN76}. Their proofs rely on the polarization identity and the one-to-one correspondence between bounded, symmetric sesquilinear forms on $\cH$ and the 
set of bounded self-adjoint operators $\cB(\cH)$ on $\cH$. We also note that the proof of Theorem \ref{tA.7} in the case where strong convergence is replaced by weak 
convergence readily follows from the corresponding scalar version (see also the
matrix-valued case studied, e.g., in \cite[Theorems\ 5.4 and 5.5]{GT00}). The various extensions from weak convergence to strong convergence in Theorem \ref{tA.7} 
then repeatedly use a standard result on monotonic sequences of bounded,
nonnegative operators in $\cH$ (called Vigier's theorem in \cite[p.\ 263]{RS90}):
\begin{align}
\begin{split}
& \text{If $0 \leq B_1 \leq B_2 \leq \cdots \leq B_\infty$, with 
$B_n, B_\infty \in \cB(\cH)$, $n\in\bbN$,}    \lb{A.52} \\
& \quad \text{then $\slim_{n\to\infty} B_n = B$ for some $B\in\cB(\cH)$.}
\end{split}
\end{align}
Similarly, recalling the extension of this convergence result to compact operators
(cf.\ \cite[Lemma\ 2.1]{AN76}):
\begin{align}
\begin{split}
& \text{If $0 \leq C_1 \leq C_2 \leq \cdots \leq C_\infty$, with 
$C_n, C_\infty \in \cB_{\infty}(\cH)$, $n\in\bbN$,}      \lb{A.53} \\
& \quad \text{then $\lim_{n\to\infty} \|C_n - C\|_{\cB(\cH)} = 0$ for
some $C\in\cB_{\infty}(\cH)$,}
\end{split}
\end{align}
repeated applications of this fact yield the extensions to $\cB(\cH)$-norm convergence 
in item $(vii)$. Of course, the monotonically increasing and uniformly bounded 
families $\{B_n\}_{n\in\bbN}$ 
and $\{C_n\}_{n\in\bbN}$ in \eqref{A.52} and \eqref{A.53} can be replaced by 
monotonically decreasing families of uniformly bounded operators in $\cH$. 
(For variations of \eqref{A.52} and \eqref{A.53} we also refer to 
\cite[Theorems\ VIII.3.3 and VIII.3.5, Remark VIII.3.4]{Ka80}.)  

In the special case of scalar Herglotz functions $m$ (cf.\ \cite{AD56} and \cite{KK74}
for detailed treatments), isolated zeros of $m$
are well-known to be necessarily simple and located on $\bbR$. This can be inferred 
from the fact that $-1/m$ is a Herglotz function whenever $m$ is one, and hence
isolated poles of $1/m$ are also necessarily simple with a negative residue. 
Studying $(f,M(z)f)_{\cH}$ for all $f \in \cH$ then illustrates item $(v)$. 

That item $(vi)$ holds, in fact, with $\slim_{\varepsilon \downarrow 0}$ rather than 
$\wlim_{\varepsilon \downarrow 0}$ follows again from monotonicity considerations: 
First, (choosing $D=0$ in \eqref{A.42} without loss of generality) one notes that the expression on the left-hand side in \eqref{A.53a} below
\begin{equation}
\varepsilon \Im\bigg(\f{1}{t - (\lambda + i \varepsilon)}\bigg) 
= \f{\varepsilon^2}{(t-\lambda)^2 + \varepsilon^2} \in [0,1], \quad (t, \lambda) \in \bbR^2, 
\; \varepsilon > 0,    \lb{A.53a}
\end{equation}
is nonnegative, uniformly bounded by $1$, and monotonically decreasing 
with respect to $\varepsilon$ as $\varepsilon \downarrow 0$. Moreover,
\begin{equation}
\lim_{\varepsilon \downarrow 0} \varepsilon 
\Im\bigg(\f{1}{t - (\lambda + i \varepsilon)}\bigg) 
= \begin{cases} 0, & t \in \bbR\backslash\{\lambda\}, \\ 
1, & t = \lambda. \end{cases}
\end{equation}
Combining this with the analog of the monotonicity result \eqref{A.52} in the 
decreasing case proves the first equality in \eqref{A.46}. In the remainder of the 
proof of item $(vi)$ we make the simplifying 
assumption that $M$ is of the form 
$M(z) = \int_{\bbR} d\Omega(\lambda) (\lambda - z)^{-1}$, $z \in \bbC_+$, 
which is permitted, without loss of generality, as only local considerations are at 
stake. Since 
\begin{equation} 
\varepsilon \Re\bigg(\f{1}{t - (\lambda + i \varepsilon)}\bigg) 
= \f{\varepsilon (t - \lambda)}{(t-\lambda)^2 + \varepsilon^2} \in [-1/2,1/2], 
\quad (t, \lambda) \in \bbR^2, \; \varepsilon > 0,
\end{equation}
is not monotonic with respect to $\varepsilon$ as $\varepsilon \downarrow 0$, we 
decompose it into three monotonic pieces as follows, 
\begin{equation}
\varepsilon \Re\bigg(\f{1}{t - (\lambda + i \varepsilon)}\bigg) 
= \psi_1(t-\lambda,\varepsilon) + \psi_2(t-\lambda,\varepsilon) -2^{-1},   \lb{A.53b}
\end{equation}
where
\begin{equation}
\psi_1(x,\varepsilon) = \begin{cases} \varepsilon x \big[x^2 + \varepsilon^2\big]^{-1}, 
& |x| \geq \varepsilon,  \\ 1/2, & |x| \leq \varepsilon, \end{cases} \quad 
\psi_2(x,\varepsilon) = \begin{cases} \varepsilon x \big[x^2+\varepsilon^2\big]^{-1}, 
& |x| \leq \varepsilon,  \\ 1/2, & |x| \geq \varepsilon. \end{cases}  
\end{equation}
By monotonicity of each of the three terms with respect to $\varepsilon$, one obtains that 
\begin{align}
& \slim_{\varepsilon \downarrow 0}\varepsilon \Re(M(\lambda + i \varepsilon) 
= \slim_{\varepsilon \downarrow 0} \int_{\bbR} d\Omega(t) 
\f{\varepsilon (t - \lambda)}{(t-\lambda)^2 + \varepsilon^2}   \no \\
& \quad = \slim_{\varepsilon \downarrow 0} 
\int_{\bbR} d\Omega(t) \big[\psi_1(t-\lambda,\varepsilon) 
+ \psi_2(t-\lambda,\varepsilon) -2^{-1} \big] = 0,   \lb{A.53c}
\end{align}
because the corresponding weak limits equal zero by the following well-known 
arguments: Let $f \in \cH$, then
\begin{align}
\begin{split} 
& \bigg|\varepsilon \int_{|t - \lambda|\geq 1} d(f,\Omega(t) f)_{\cH} 
\f{(t - \lambda)}{(t - \lambda)^2 + \varepsilon^2}\bigg|    \\
& \quad \leq \varepsilon 
\int_{|t - \lambda|\geq 1} d(f,\Omega(t) f)_{\cH} |t - \lambda|^{-1} 
 \underset{\varepsilon \downarrow 0}{\longrightarrow} 0.  
\end{split} 
\end{align}
By polarization, also
\begin{equation}
\lim_{\varepsilon \downarrow 0} \bigg|\varepsilon \int_{|t - \lambda|\geq 1} 
d(f,\Omega(t) g)_{\cH} \f{(t - \lambda)}{(t - \lambda)^2 + \varepsilon^2}\bigg| 
= 0, \quad f, g \in \cH. 
\end{equation} 
Next, for $f \in \cH$,
\begin{align} 
\begin{split} 
& \lim_{\varepsilon \downarrow 0} \bigg|\int_{|t - \lambda|\leq 1} d(f,\Omega(t) f)_{\cH} 
\f{\varepsilon (t - \lambda)}{(t - \lambda)^2 + \varepsilon^2}\bigg|    \\
& \quad \leq \lim_{\varepsilon \downarrow 0} \int_{|t - \lambda|\leq 1} d(f,\Omega(t) f)_{\cH} 
\f{\varepsilon |t - \lambda|}{(t - \lambda)^2 + \varepsilon^2} = 0,  
\end{split} 
\end{align}
applying the dominated convergence theorem, as 
\begin{equation}
\f{\varepsilon |t - \lambda|}{(t - \lambda)^2 + \varepsilon^2} \leq \f{1}{2}, 
\quad t \in \bbR, \; \varepsilon > 0.   
\end{equation}
Again by polarization, 
\begin{equation}
\lim_{\varepsilon \downarrow 0} \bigg|\varepsilon \int_{|t - \lambda|\leq 1} 
d(f,\Omega(t) g)_{\cH} \f{(t - \lambda)}{(t - \lambda)^2 + \varepsilon^2}\bigg| 
= 0, \quad f, g \in \cH,
\end{equation} 
completing the proof of  
\begin{equation}
\wlim_{\varepsilon \downarrow 0}\varepsilon \Re(M(\lambda + i \varepsilon)) = 0. 
\end{equation}
Thus, \eqref{A.53c} together with the first equality in \eqref{A.46}, 
then also prove the second equality in \eqref{A.46} and hence completes the proof of 
item $(vi)$.

Item $(viii)$ is a consequence of \cite[Subsections\ 1.2.4 and 1.2.5]{BW83} (which in 
turn are based on \cite{As67}). 
$\hspace*{10cm} \square$

\smallskip

As usual, the normal limits in Theorem \ref{tA.7} can be replaced by nontangential 
ones. 

The nature of the boundary values of $M(\cdot + i 0)$ when for some $p>0$, 
$M(z) \in \cB_p(\cH)$, $z \in \bbC_+$, was clarified in detail in \cite{BE67}, 
\cite{Na89}, \cite{Na90}, \cite{Na94}.  

Using an approach based on operator-valued Stieltjes integrals, a special case
of Theorem \ref{tA.7} was proved by Brodskii \cite[Sect.\ I.4]{Br71}. In particular,
he proved the analog of the Herglotz representation for operator-valued 
Caratheodory functions. More precisely, if $F$ is analytic on $\dD$ with nonnegative 
real part $\Re(F(w)) \geq 0$, $w \in \D$, then $F$ is of the form
\begin{align}
\begin{split}
& F (w) = i \Im(F(0)) + \oint_{\dD} d\Upsilon (\zeta) \, \f{\zeta + w}{\zeta - w},
\quad w \in \D,      \lb{A.54} \\
& \Re (F (0)) = \Upsilon(\dD),
\end{split}
\end{align}
with $\Upsilon$ a bounded, nonnegative $\cB(\cH)$-valued measure on $\dD$. The result \eqref{A.54} can also be derived by an application of Naimark's dilation
theory (cf.\ \cite{AN76} and \cite[p.\ 68]{Fi70}), and it can also be used to derive the Nevanlinna representation \eqref{A.41a}, \eqref{A.42} (cf.\ \cite{AN76}, and in a 
special case also \cite[Sect.\ I.4]{Br71}). Finally, we also mention that 
Shmuly'an \cite{Sh71} discusses the Nevanlinna representation \eqref{A.41a}, 
\eqref{A.42}; moreover, certain special classes of Nevanlinna functions, isolated 
by Kac and Krein \cite{KK74} in the scalar context, are studied by Brodskii \cite[Sect.\ I.4]{Br71} and Shmuly'an \cite{Sh71}.

For a variety of applications of operator-valued Herglotz functions, see, for 
instance, \cite{AL95}, \cite{ABMN05}, \cite{ABT11}, \cite{BMN02}, \cite{Ca76}, 
\cite{DM91}--\cite{DM97}, \cite{GKMT01}, 
\cite{MM02}--\cite{MN11}, \cite{Sh71}, and the literature cited therein.

\medskip

{\bf Acknowledgments.} We are indebted to Mark Malamud for a critical reading 
of our manuscript and for numerous helpful suggestions. We are also grateful 
to Joe Diestel for his help in clarifying the role of the Radon-Nikodym property in 
connection with vector-valued absolutely contionuous functions.   


\end{document}